\documentclass[a4paper,reqno,11pt,oneside]{amsart}
\usepackage{amsmath}
\usepackage[left=2.61cm,right=2.61cm,top=2.72cm,bottom=2.72cm]{geometry}
\usepackage[colorlinks=true,urlcolor=blue,
citecolor=red,linkcolor=blue,linktocpage,pdfpagelabels,
bookmarksnumbered,bookmarksopen]{hyperref}
\usepackage[english]{babel}
\usepackage{graphicx}
\usepackage[toc,page,header]{appendix}
\usepackage{mathrsfs}
\usepackage{amssymb,amsmath, amsfonts, amsthm}
\usepackage{latexsym}
\usepackage{enumitem}
\allowdisplaybreaks[4]

\newcommand{\Ric}{\operatorname{Ric}}

\newcommand{\R}{\mathbb{R}}
\numberwithin{equation}{section}

\newtheorem{theorem}{Theorem}[section]

\newtheorem{corollary}[theorem]{Corollary}
\newtheorem{lemma}[theorem]{Lemma}
\newtheorem{remark}[theorem]{Remark}
\newtheorem{definition}[theorem]{Definition}

\hypersetup{
  colorlinks=true,
  linkcolor=magenta,
  citecolor=blue,
  filecolor=blue,
  urlcolor=blue
}

\begin{document}

\title[\small Gradient estimates for $p$-Laplacian equation]{Gradient estimates for $p$-Laplacian equation with  cubic polynomial nonlinearity on  Riemannian manifolds}

\author{Zhen Qiu}
\address[Zhen Qiu]{School of Mathematics and Information Science, Guangzhou University, Guangzhou 510006, China.}
\email{qiuzhen97@foxmail.com}

\author{Youde Wang}
\address[Youde Wang]{1. School of Mathematics and Information Sciences, Guangzhou University
2. State Key Laboratory of Mathematical Sciences (SKLMS), Academy of Mathematics and Systems Science, Chinese Academy of Sciences, Beijing 100190, China
3. School of Mathematical Sciences, University of Chinese Academy of Sciences, Beijing 100049, China.}
\email{wyd@math.ac.cn}

\author{Jun Yang}
\address[Jun Yang]{School of Mathematics and Information Science, Guangzhou University, Guangzhou 510006, China.}
\email{jyang2019@gzhu.edu.cn}


\begin{abstract}

This paper studies a class of $p$-Laplace equations with  cubic polynomial nonlinearity
\[
\Delta_p v + (v-a_1)(v-a_2)(v-a_3) = 0
\]
on complete Riemannian manifolds $M$ with lower Ricci curvature bounds,
where $a_1 < a_2 < a_3$ are real constants and $\Delta_p v = \operatorname{div}(|\nabla v|^{p-2}\nabla v)$ denotes the $p$-Laplace operator.
Depending on whether the solution lies in the intervals $(a_1,a_2), (a_2,a_3)$ or $(a_1,a_3)$,
we employ, respectively, a logarithmic transformation or a hyperbolic tangent transformation to convert the original equation to another one
for further analysis.
Through a detailed analysis of the lower-bound estimate for the linearized operator of the new equation,
and by combining Saloff-Coste's Sobolev inequality with a Moser iteration, we establish Cheng-Yau type gradient estimates under an additional assumption on $p$.
As applications, the Liouville theorem and a Harnack inequality are further proved.
\end{abstract}


\thanks{
Zhen Qiu  is supported by the innovation research for Postgraduate of Guangzhou University(No. JCCX2024-013) and NSFC(No. 12201140). Youde Wang is supported by NSFC(No. 12431003). Jun Yang is supported by NSFC(No. 12171109). Corresponding author: Jun Yang, jyang2019@gzhu.edu.cn}
\maketitle
\tableofcontents

\section{Introduction}
In this paper, we are concerned with a class of $p$-Laplacian equations with cubic polynomial nonlinearity:
\begin{equation}\label{GAC}
\Delta_p v + (v-a_1)(v-a_2)(v-a_3) = 0, \quad {\rm in} \,\: M,\tag{PCN}
\end{equation}
where $(M, g)$ is a complete Riemannian manifold equipped with a Riemannian metric $g$, and $a_1 < a_2 < a_3 $ are real numbers. Here, the $p$-Laplacian on $(M, g)$ is defined by
$$
\Delta _p v= \mathrm{div} (|\nabla  v|^{p-2} \nabla v),\quad \forall\,  v \in W^{1,p},
$$
which is understood in the distributional sense. In particular, when $p=2$, it is exactly the standard Beltrami-Laplacian.

\subsection{ Background and Motivations}\
The gradient estimate  that combines curvature conditions with the maximum principle has dominated research on gradient estimates for geometric partial differential equations on manifolds over the past half-century. Yau \cite{MR431040} and  Cheng and Yau \cite{MR385749} established the gradient estimate, i.e., the so-called ``Cheng-Yau type gradient estimate", for the positive solution of the  harmonic function equation
\begin{align}\label{harmonic equation}
    \Delta v =0
\end{align}
on complete Riemannian manifolds with Ricci curvature bounded from below.
Suppose $(M,g)$ is an $n$-dimensional complete noncompact Riemannian manifold with Ricci curvature $\Ric_g(M)\geq -(n-1)\kappa g$, the Cheng-Yau type gradient estimate of \eqref{harmonic equation} is precisely formulated as
\begin{align*}
\sup\limits_{B_{R/2}(o)} \frac{|\nabla v|^2}{v^2}
\leq C(n) \left(\frac{1+ \sqrt{\kappa} R }{R}\right)^2, \quad \forall\, B_{R/2}(o)\subset M,
\end{align*}
which provides profound geometric and regularity characterizations of solutions. For the above harmonic function equation, the gradient estimate exhibits an extremely concise form. The elegance of this estimate lies in the fact that its expression depends only on three fundamental quantities: the dimension $n$ of the manifold, the curvature parameter $ \kappa$, and the radius $R$ of the domain. It is precisely this simplicity that makes it a highly general and widely applicable result.

It is worth emphasizing that such a gradient estimate carries powerful analytical implications.
From it, one can directly derive the classical Harnack inequality (an inequality that describes the preservation of nonnegativity and the control of upper and lower bounds of solutions in local regions, reflecting the regularity of solutions and the strong maximum principle). Furthermore, the gradient estimate provides a natural approach to  Liouville-type theorems: under whole-space or boundedness conditions, the growth constraint on the gradient implies that the solution must be constant. Thus, the gradient estimate for the harmonic equation is not only concise in form and minimal in its dependence on parameters, but also analytically powerful, serving as a key bridge connecting geometry, analysis, and partial differential equations.

Over the past two decades, many authors have used similar techniques to prove gradient estimates, Liouville theorem and Harnack inequalities for $p$-Laplacian equation:
\begin{align*}
\Delta_p  v=0.
\end{align*}
Kotschwar and Ni \cite{MR2518892} proved the gradient estimates of the $p$-Laplacian equation under the assumption that the sectional curvature of the complete Riemannian manifold is bounded from below by a constant. Later on, Wang and Zhang \cite{MR2880214} adopted Nash-Moser iteration to establish gradient estimates for positive $p$-Laplacian functions under weaker curvature condition, i.e., the Ricci curvature is bounded from below.
Wang and Zhang's work is vital, reflected not only  in the replacement of sectional curvature by Ricci curvature but also in the introduction of the powerful tool, i.e., the Nash-Moser iteration.

On the other hand, many mathematicians paid attentions to the following semilinear elliptic equation named after Lane-Emden on a Riemannian manifold $M$ or an Euclidean space $\mathbb{R}^n$
\begin{align}\label{LE}
\Delta v + v^q =0.
\end{align}
For instances, Gidas and Spruck \cite{MR615628} proved the Liouville theorem in the case $1\leq q <\frac{n+2}{n-2}$ for $n\geq 2$. Recently, Wang and Wei \cite{MR4559367} adopted the Nash-Moser iteration scheme to prove a local Cheng-Yau type gradient estimate, which is of the same form as in \cite{MR385749}, for positive solutions to Lane-Emden equation if
$$q \in \left(-\infty,\quad \frac{n+1}{n-1} + \frac{2}{\sqrt{n(n-1)}}\right).$$
This gradient estimate also reveals the non-existence of the positive weak solution to the above Lane-Emden equation provided the Ricci curvature of $M$ is nonnegative and hence improves the Liouville theorem due to Gidas and Spruck. Later on, for the case $q\in (-\infty,\, \frac{n+2}{n-2})$ in \eqref{LE} Lu \cite{Lu} established the local Cheng-Yau type logarithm gradient estimate for positive solutions to Lane-Emden equation \eqref{LE} on a complete Riemannian manifold with Ricci curvature bounded from below.

One also considered superlinear elliptic problems on $\mathbb{R}^n$
\begin{align}\label{1.3}
\Delta_p v + h(v) = 0.
\end{align}
In the case $h(v)= v^q$ this equation reduces to Lane-Emden-Fowler equation
\begin{align}\label{LEF}
\Delta_p v + v^q =0.
\end{align}
Serrin and Zou in \cite{MR1946918} proved that if $1<p<n$ and $q>0$, then equation \eqref{LEF} defined on $\R^n$ admits no positive solution if and only if
$$ 0<q<np/(n-p)-1.$$
Recently, inspired by Wang and Zhang \cite{MR2880214}, He, Wang and Wei \cite{MR4703505} considered
\begin{align}\label{1.5}
\Delta_p v + av^q =0 \quad\mbox{on}\,\, M
\end{align}
and used Moser iteration to gain the logarithm gradient estimate of the positive solution and extend largely the range of $q$ for Liouville theorem which combines with the result of Serrin and Zou \cite{MR1946918} to imply the optimal Liouville result for equation \eqref{LEF} defined on $\mathbb{R}^n$. More precisely, they showed that there is no positive solution of \eqref{1.5} defined on a complete Riemannian manifold of nonnegative Ricci curvature satisfying
$$a>0 \quad \&\quad  q<\frac{n+3}{n-1}(p-1)\quad \text{or} \quad a<0 \quad \& \quad  q>p-1.$$	
It should be mentioned that Huang, Guo and Guo \cite[Theorem 1]{MR4676651} posted a weaker result with additional restriction on $p$ which is not required in \cite{MR4703505}. Shortly after, He, Sun and Wang \cite{HSW} proved the optimal Liouville theorem for equation \eqref{LEF} defined on a noncompact Riemannian manifold $M$ with nonnegative Ricci curvature.

In a recent preprint, Wang,  Wei, and  Zhang \cite{WangWeiZhang} investigate the Liouville property and gradient estimates for positive solutions \eqref{LEF} on complete noncompact Riemannian manifolds. A key innovation is the removal of a prior volume growth assumption. The authors establish a volume growth estimate for manifolds satisfying a $\chi$-type Sobolev inequality, showing that $\mathrm{Vol}(B_r) \ge C r^{2\chi/(\chi-1)}$, which is of independent geometric interest. By employing a Nash-Moser iteration technique and a meticulous analysis of the linearized operator, they prove that if the $L^{\chi/(\chi-1)}$-norm of the negative part of the Ricci curvature is sufficiently small relative to the Sobolev constant and the volume growth order, then no positive solution exists for a wide range of parameters $p$ and $q$, including the harmonic ($a=0$) and Lane-Emden ($p=2$, $a=1$) cases. Furthermore, under the weaker condition that $\mathrm{Ric}_-$ is in $L^\gamma$ for some $\gamma > \chi/(\chi-1)$, they derive a local logarithmic gradient estimate. As a striking geometric application of their analytic results, they prove a gap theorem for the number of ends: a complete manifold of dimension $n\ge 3$ satisfying a Sobolev inequality and with $\|\mathrm{Ric}_-\|_{L^{n/2}}$ smaller than a constant depending only on $n$ and the Sobolev constant must have only one end.
\medskip

On the other hand, Pol\'a\v{c}ik, Quittner and Souplet \cite{PQS} studied some new connections between Liouville-type theorems and local properties of nonnegative solutions to the above superlinear elliptic problem \eqref{1.3}. Namely, they develop a general method for derivation of universal and pointwise {\it a priori} estimates of local solutions from Liouville type theorems, which provides a simpler and unified treatment for such questions. Concretely, they proved the following:

\smallskip
{\it
Let $0< p-1< \alpha <p^*-1$, where $p^*$ is the $p$ critical exponent, and assume that
$$\lim_{v\to+\infty}v^{-\alpha} h(v)=\ell \in(0,+\infty).$$
\noindent
Let $\Omega$ be an arbitrary domain of $\mathbb{R}^n$. Then there exists $C = C(p, n, h) > 0$ (independent of $\Omega$ and $v$) such
that for any (nonnegative) solution $v$ to $\Delta_p v + h(v) =0$ in $\Omega$, there holds
$$
v + |\nabla v|^{p/(\alpha+1)} \leq C\Big(1 + \big|dist(x,\partial\Omega)\big|^{-p/(\alpha+1-p)}\Big), \quad x \in \Omega.
$$
\noindent
In particular, if $\Omega = B_R \setminus \{0\}$ for some $R > 0$, then
$$v + |\nabla v|^{p/(\alpha+1)} \leq C\Big(1 + |x|^{-p/(\alpha+1-p)}\Big),\quad 0 < |x| \leq\frac{R}{2}.$$
}

Very recently, He, Ma and Wang \cite{MR4878821} considered the case $h(v)= bv^{q}+cv^{r}$ in \eqref{1.3} and established the local and global Cheng-Yau type gradient estimate of the positive solution, Liouville theorem and Harnack inequality for a class of $p$-Laplace equations defined on $(M^n, g)$ with $\Ric_g(M)\geq -(n-1)\kappa g$ written by
\begin{align*}
\begin{cases}
\Delta_{p}v+bv^{q}+cv^{r}=0,&\text{in }M\,;\\[2mm]
v>0,&\text{in }M\,;\\[2mm]
b,c\in\mathbb{R},\ p>1,\ q\leq r, \: n\geq 2,
\end{cases}
\end{align*}
where the parameters $(b,c,q,r)$ belong to the regions defined as follows:
\begin{align*}
W_{1} &=\left\{(b,c,q,r): b\left(\frac{n+1}{n-1}-\frac{q}{p-1}\right)\geq 0 \quad\text{and}\quad c\left(\frac{n+1}{n-1}-\frac{r}{p-1}\right)\geq 0\right\},
\\[2mm]
W_{2} &=\left\{(b,c,q,r):c\leq 0 \quad\text{and}\quad \left|\frac{q}{p-1}-\frac{n+1}{n-1}\right|<\frac{2}{n-1}\right\},
\\[2mm]
W_{3} &=\left\{(b,c,q,r):b\geq 0 \quad\text{and}\quad \left|\frac{r}{p-1}-\frac{n+1}{n-1}\right|<\frac{2}{n-1}\right\}.
\end{align*}
As direct consequences, they deduced that if $\kappa =0$ then $v\equiv1$ is the unique positive solution to the generalized Allen-Cahn equation:
\begin{align*}
\Delta_{p}v+v-v^{3}=0, \quad \dfrac{2n+2}{n+3} <p <4,
\end{align*}
and the generalized Fisher-KPP equation:
\begin{align*}
\Delta_{p}v+v-v^{2}=0,\quad\dfrac{2n+2}{n+3} <p <3.
\end{align*}

Regarding the famous  Allen-Cahn equation
\begin{align*}
\Delta v+v-v^{3}=0,
\end{align*}
which originates from the gradient theory of phase transitions and thus has an intricate connection to minimal surface theory, Hou \cite{MR3894900} derived gradient estimates of bounded positive smooth solutions from two cases, i.e., $ v\leq C \leq 1$ and $v \leq C $ with $C>1$ respectively, on the complete Riemannian manifolds with Ricci curvature bounded below $\mathrm{Ric}_g \geq K  g$ with $ K \geq 0$. Lu \cite{MR4735218} first obtained a logarithmic gradient estimate for the local positive solutions of Allen-Cahn equation under the same assumption on Ricci curvature.

We also refer to some papers on gradient estimate by  Nash-Moser iteration, please check  \cite{MR4951166, MR4938647, MR4892200, MR4828200, MR4748745, MR4957916, MR4882814, MR4867213} for several class of  quasilinear elliptic equations, \cite{MR4874664, MR4881820, MR4838045, MR4939105} for the semilinear elliptic equations and \cite{MR4841602,MR4887175} for some nonlinear parabolic equations. For further discussion on critical $p$-Laplacian equation, we refer the reader to  Catino, Monticelli and Roncoroni \cite{MR4649289}, Ou \cite{ou},  V\'etois \cite{MR4747988, MR4194414, MR3411668},  Catino and Monticeli \cite{onlinepaper}, Fogagnolo, Malchiodi and Mazzieri \cite{MR4567579}, and the references therein.
\medskip

Our motivation in this article is to consider the $p$-Laplacian equation with the polynomial nonlinearity $h(v)$ and establish Cheng-Yau type gradient estimates for its solutions without the restriction of positivity.
For simplicity we choose
$$h(v)=(v-a_1)(v-a_2)(v-a_3)$$
as a toy model. More precisely, we will derive the gradient estimates, Harnack inequalities and Liouville properties for the solutions to \eqref{GAC} with $a_1 < a_2 < a_3$.

\subsection{ Main results}\

Before stating the main results in this paper, we would like to give the definition of the weak solution to the $p $-Laplacian equation with cubic polynomial nonlinearity \eqref{GAC}.

\begin{definition}\label{definition}
The function $v$ is called the weak solution of the $p $-Laplacian equation with  cubic polynomial nonlinearity \eqref{GAC} on $\Omega$  if $v \in C^{1} (\Omega) \cap W^{1,p}_{loc} (\Omega)$  and satisfies the following equality
\begin{align*}
-\int_{\Omega} |\nabla v|^{p-2} \langle \nabla v, \nabla \psi\rangle + \int_{\Omega} (v-a_1)(v-a_2)(v-a_3)  \psi =0,\quad \forall\, \psi \in W^{1,p}_0(\Omega).
\end{align*}
\end{definition}

The first theorem is the classical Cheng-Yau type gradient estimate for  $p $-Laplacian equation with  cubic polynomial nonlinearity \eqref{GAC}.
\begin{theorem}\label{theorem1}
Let $(M,g) $ be an $n$-dimensional complete Riemannian manifold with $\mathrm{Ric}_g \geq -(n-1) \kappa g$, where $n\geq 2$ and $\kappa$ is a non-negative constant. Suppose that $v$ is a solution in the sense of Definition \ref{definition} to the $p $-Laplacian equation with cubic polynomial nonlinearity \eqref{GAC} on the geodesic ball $B_R(o) \subset M$ with radius $R$ and center at $o\in M$.
\begin{itemize}
\item Suppose $a_1 <v <a_2$  or $a_2< v< a_3$, and
\begin{align}\label{CY1}
p \in \left( \frac{2n+2}{n+3},\,\,\, 4\right)
\end{align}
then there holds the Cheng-Yau type  gradient estimate
\begin{align}
\sup\limits_{B_{R/2}(o)} \frac{|\nabla v|^2}{|v-a_1|^2}
\leq
C(n,p) \left(\frac{1+ \sqrt{\kappa} R }{R}\right)^2, \quad
{\rm  for } \: a_1 <v <a_2,
\label{Cheng-Yau1}
\end{align}
and
\begin{align}
\sup\limits_{B_{R/2}(o)} \frac{|\nabla v|^2}{|a_3-v|^2}
\leq
C(n,p) \left(\frac{1+ \sqrt{\kappa} R }{R}\right)^2, \quad
{\rm  for } \: a_2 <v <a_3.
\label{Cheng-Yau2}
\end{align}

\item Suppose $a_1 < v < a_3$  and
\begin{align}\label{CY2}
p\in \left[\dfrac{2n}{n+1}, \,\,\,\, \frac{2n}{n+1} + \frac{(n-1)(1 +\sqrt{1-\delta^2})}{(n+1)\delta^2} \right]\cap\:(1, \infty)
\end{align}
where $\delta := \frac{a_1 + a_3 - 2a_2}{a_3 - a_1}\in (-1,1)$. Then there holds the Cheng-Yau type  gradient estimate
\begin{align}
\sup\limits_{B_{R/2}(o)}  \frac{|\nabla v|^2}{4[(v - a_1)(a_3 - v)]^2}
\leq
\dfrac{C(n,p)}{(a_3 - a_1)^2} \left(\frac{1+ \sqrt{\kappa} R }{R}\right)^2.
\label{Cheng-Yau3}
\end{align}
\end{itemize}
\end{theorem}

\smallskip
\begin{remark}\label{remark1.1}
In previous papers such as \cite{MR4703505}, the logarithmic transformation plays a key role for the proofs of the Cheng-Yau type gradient estimates and other consequent results therein. We will follow the same method in the proofs of \eqref{Cheng-Yau1} and \eqref{Cheng-Yau2}, see Lemma \ref{lem:gradient_estimate_bounded} and Corollary \ref{corollary2.1}. The logarithmic transform  is effective because it converts power-type nonlinearities involving $v$ into exponential ones, which are convex functions of $w$ given by \eqref{transformation-log} or \eqref{transformation-log2}. Convexity ensures that linear combinations of different exponential terms, after completing squares in the nonlinear terms, yield a dominant positive term under suitable conditions, which is crucial for the lower bound estimates in Section \ref{section2.1}.

However, the effect of the standard logarithmic transformation fails in the proof of \eqref{Cheng-Yau3}. To recover the method, in Section \ref{section2.1} we will need to introduce the hyperbolic tangent transformation in \eqref{transformation-tangent}. Whence, Section \ref{section2.2} is the new ingredient of the present paper.
\end{remark}

Moreover, by applying the gradient established in the above we can also obtain some Liouville theorems and Harnack inequalities for the $p $-Laplacian equation with cubic polynomial nonlinearity \eqref{GAC} if the Ricci curvature of $(M,g)$ is bounded from below.

\begin{theorem}\label{theorem3}
Let $(M,g) $ be an $n$-dimensional non-compact and complete Riemannian manifold with non-negative Ricci curvature and $n\geq 2$.
\begin{itemize}
\item Suppose that p satisfies \eqref{CY1}, then the $p $-Laplacian equation with  cubic polynomial nonlinearity \eqref{GAC} admits no solution $ v$ which lies in  $(a_1,a_2)$ or $ (a_2,a_3)$.
\item Suppose that $p$ satisfies \eqref{CY2}. If the $p $-Laplacian equation with  cubic polynomial nonlinearity \eqref{GAC} admits a solution $ v$ which lies in $ (a_1,a_3)$, then $ v \equiv a_2.$
\end{itemize}
\end{theorem}

\begin{theorem}\label{theorem4}
Let $(M,g) $ be an $n$-dimensional non-compact and complete Riemannian manifold with $\mathrm{Ric}_g \geq -(n-1) \kappa g$
where $\kappa$ is a nonnegative constant and $n\geq 2$, and  $ v$ be a  solution in the sense of Definition \ref{definition} to  the $p $-Laplacian equation with cubic polynomial nonlinearity \eqref{GAC}.
\begin{itemize}
\item If $a_1< v < a_2$ and $p$ satisfies \eqref{CY1}, then there holds the Harnack inequality:
\begin{align}
\dfrac{v(x)-a_1}{v(y)-a_1}\leq e^{C(n,p) (1 + \sqrt{\kappa} R)},\quad  \forall \:  x, y \in B_{R\slash 2}(x_0),\  \forall\, x_0\in M.
\label{Harnack1}
\end{align}

\item If $a_2< v < a_3$ and $p$ satisfies \eqref{CY1}, then there holds the Harnack inequality:
\begin{align}
        \dfrac{a_3- v(x)}{a_3- v(y)}
        \leq e^{C(n,p) (1 + \sqrt{\kappa} R)},\quad  \forall \: x, y \in B_{R\slash 2}(x_0),\  \forall\, x_0\in M.
        \label{Harnack2}
\end{align}

\item If $a_1 < v < a_3$ and $p$ satisfies \eqref{CY2}, there holds the Harnack inequality:
\begin{align}
\frac{\big(v(x)-a_1\big)\big(a_3-v(y)\big)}{\big(v(y)-a_1\big)\big(a_3-v(x)\big)} \leq e^{C(n,p) (1 + \sqrt{\kappa} R)},\quad  \forall \: x, y \in B_{R\slash 2}(x_0),\  \forall\, x_0\in M.
\label{Harnack3}
\end{align}
\end{itemize}
\end{theorem}

\smallskip
\begin{remark}\label{remark1.2}
Notice that $(a_1,a_2) \subset (a_1,a_3)$ and $(a_2,a_3)\subset (a_1,a_3)$, and then the  Cheng-Yau type gradient estimate \eqref{Cheng-Yau3} in \autoref{theorem1} or the Harnack inequality \eqref{Harnack3} in \autoref{theorem4} is also true when $ v $ locates in $ (a_1,a_2) $ or $ (a_2,a_3)$ together with the validity of the assumption in \eqref{CY2}.
\end{remark}

The rest of this paper is organized as follows.
In Section \ref{lowerbound}, we will first transform the solution $v$ of \eqref{GAC} to another function $w$
in \eqref{transformation-log}, or \eqref{transformation-log2}, or \eqref{transformation-tangent}, and then derive the lower-bound estimates of $\mathcal{L} (|\nabla w |^{2\alpha})$ where ${\mathcal L}$ is the linearization of the $p$-Laplacian at $w$ and has also been given in \eqref{linearoperator}. The readers can refer to Lemma \ref{lem:gradient_estimate_bounded}, Corollary \ref{corollary2.1} and Lemma \ref{lemma3.5.2}.
In Section \ref{sec2.2}, by choosing $\alpha$ large enough and using Saloff-Coste's Sobolev inequality, we establish an integral inequality, which enables the local $L^\gamma$-upper bound estimate of the gradient to the weak solution to \eqref{GAC} in Section \ref{sec2.3}.
In Section \ref{sec2.4}, we start the Moser iteration to obtain the  gradient estimate in  Theorem \ref{theorem1}.
As  applications, we prove the Liouville theorem and Harnack inequality in Theorems \ref{theorem3} and \ref{theorem4}, see Section \ref{sec2.5}.

\section{Lower-bound estimate of $\mathcal{L}(f^{\alpha}) $ for bounded solutions}\label{lowerbound}

The linearization operator $\mathcal{L}$ of the $p$-Laplacian at any function $w$ will be denoted by
\begin{align}
\mathcal{L}(\psi) = \mathrm{div} \left( f^{p/2 -1}  \mathcal{A} (\nabla \psi)\right),
\label{linearoperator}
\end{align}
where
\begin{align}\label{definition-f}
f = |\nabla w|^2 \qquad\mbox{and}\qquad \mathcal{A}(\nabla \psi) = \nabla \psi  + (p-2) f^{-1}\langle\nabla\psi, \nabla w\rangle\nabla w.
\end{align}
Recall the identity from \cite[Lemma 2.3]{MR4703505}:
\begin{equation}\label{lemma3.1.3}
\begin{split}
\mathcal{L}(f^\alpha) = &\alpha\left(\alpha + \frac{p}{2} - 2\right) f^{\alpha + \frac{p}{2} - 3} |\nabla f|^2 + 2\alpha f^{\alpha + \frac{p}{2} - 2} \left( |\nabla \nabla w|^2 + \operatorname{Ric}(\nabla w, \nabla w) \right)
\\[2mm]
&+ \alpha(p-2)(\alpha - 1) f^{\alpha + \frac{p}{2} - 4} \langle \nabla f, \nabla w \rangle^2 + 2\alpha f^{\alpha - 1} \langle \nabla \Delta_p w, \nabla w \rangle,
\end{split}
\end{equation}
where $\alpha > 0$ is any constant.

\subsection{ The low-bound under the logarithmic transformation}\label{section2.1}\

In previous papers such as \cite{MR4703505}, the logarithmic transformation was chosen as a key step for the proofs of the results therein,
and we follow the same method in the proofs of Lemma \ref{lem:gradient_estimate_bounded} and Corollary \ref{corollary2.1} below.
However, in \autoref{remark2.1}, we will reveal that this transformation is incompetent to deal with another case, which will be concerned in Section \ref{section2.2}.

\begin{lemma}\label{lem:gradient_estimate_bounded}
Suppose $v$ is a solution (in the sense of \autoref{definition}) with the property $a_1 <v <a_2$, and
$p\in \left(\frac{2n+2}{n+3},\: 4\right)$, then there exist two constants $\alpha_0 = \alpha_0(n, p) > \frac{3}{2}$ and $\beta_{n,p,\alpha_0}>0$ defined in \eqref{constant3} such that the following inequality holds pointwisely on $\{x : f(x) > 0\}$:
\begin{align}\label{lowerbound1}
\dfrac{f^{2-\alpha_0-\frac{p}{2}}}{2\alpha_0}\mathcal{L}(f^{\alpha_0}) \geq \beta_{n,p,\alpha_0} f^{2} - (n-1)\kappa f - \dfrac{\mathfrak{C}_1}{2} |\nabla f| f^{ \frac{1}{2}},
\end{align}
where $\mathfrak{C}_1 = \left| p - \frac{2(p-1)}{n-1} \right|$, and $f = |\nabla w|^2$ with
\begin{align}
w=-(p-1)\ln \big(v-a_1\big).
\label{transformation-log}
\end{align}
\end{lemma}

\begin{proof}
\textbf{Step 1. Equation Transformation.} \\
Let $u = v - a_1$, then $  0<u<b $ with $b = a_2 - a_1 > 0$ and $c = a_3 - a_1 > b$,
\[
\Delta_p u + u(u - b)(u - c) = 0.
\]
Applying the logarithmic transformation
$$
u = e^{-\frac{w}{p-1}},\qquad\mbox{i.e.,}\quad w = -(p-1)\ln u,
$$
we obtain
\begin{equation}\label{lemma3.1.1}
\Delta_p w - |\nabla w|^p - (p-1)^{p-1} e^w \left[ e^{-\frac{3w}{p-1}} - (b + c)e^{-\frac{2w}{p-1}} + bc\, e^{-\frac{w}{p-1}} \right] = 0.
\end{equation}
Define a function
$$
f = |\nabla w|^2
$$
and the following constants
\begin{equation}\label{constant1}
\begin{aligned}
&b_1 = (p-1)^{p-1}, \quad c_1 = -(b + c)(p-1)^{p-1}, \quad d_1 = bc(p-1)^{p-1},\\[2mm]
&q_1 = 1-\frac{3}{p-1}, \quad r_1 = 1-\frac{2}{p-1}, \quad s_1 = 1-\frac{1}{p-1}.
\end{aligned}
\end{equation}
We rewrite the equation \eqref{lemma3.1.1} as
\begin{equation}\label{lemma3.1.2}
\Delta_p w - f^{p/2} - b_1 e^{q_1 w} - c_1 e^{r_1 w}- d_1 e^{s_1 w} = 0.
\end{equation}
\medskip

\noindent
\textbf{Step 2. Pointwise Estimate for $\mathcal{L}(f^\alpha)$.}\\
In order to compute the terms in \eqref{lemma3.1.3} with $w$ satisfying \eqref{lemma3.1.2},
we choose $\{e_1, e_2, \cdots, e_n \}$ as an orthonormal frame of  $TM$ on a domain with $f \neq 0$ such that $e_1 = {\nabla w}/{|\nabla w|}$.

\smallskip
\noindent$\clubsuit$\
By some computations, we  first give  the following facts
\begin{align}\label{3.1}
w_1 := \frac{\partial w}{\partial {e_1} } = f^\frac{1}{2},
\qquad
w_{11} = \partial_{e_1} f^\frac{1}{2} = \frac{1}{2} f^{-\frac{1}{2}} \partial_{e_1}f
        =\frac{1}{2} f^{-1} \langle \nabla w, \nabla f\rangle.
    \end{align}

\smallskip
\noindent$\clubsuit$\
By using $w_1 = f^\frac{1}{2}$, we derive
\begin{align*}
    \frac{|\nabla f|^2 }{f} = 4 \sum_{i=1}^n w_{1i}^2,
\end{align*}
which gives
\begin{align*}
    \frac{|\nabla f|^2 }{f} \geq 4w_{11}^2,
\end{align*}
and
\begin{align}\label{3.4}
|\nabla \nabla w|^2\geq \sum_{i=1}^n w_{1i}^2 +\sum_{i=2}^n w_{ii}^2\geq\frac{|\nabla f|^2 }{4f} + \frac{1}{n-1} \left(\sum_{i=2}^n w_{ii}\right)^2.
\end{align}
With respect to the above frames, $\Delta_p w$ owns the following expression that (check in \cite{MR2518892, MR834612})
\begin{align*}
\Delta_p  w= f^{\frac{p}{2}-1}\left((p-1) w_{11} + \sum_{i=2}^n w_{ii}\right).
\end{align*}
And then, substituting this expression into \eqref{lemma3.1.2},  it gives
\begin{align*}
(p-1) w_{11} + \sum_{i=2}^n w_{ii} = f +\big[b_1 e^{q_1 w} + c_1 e^{r_1 w} + d_1 e^{s_1 w}\big]f^{1-\frac{p}{2}}
\end{align*}
and thus,
\begin{align}\label{lemma3.1.6}
\left(\sum_{i=2}^n w_{ii}\right)^2
= &\left(f +\big[b_1 e^{q_1 w} + c_1 e^{r_1 w} + d_1 e^{s_1 w}\big]f^{1-\frac{p}{2}} - (p-1) w_{11}\right)^2\nonumber\\[2mm]
= & f^2 + \big[b_1 e^{q_1 w} + c_1 e^{r_1 w} + d_1 e^{s_1 w}\big]^2 f^{2-p} + (p-1)^2 w_{11}^2\nonumber\\
&+ 2 \big[b_1 e^{q_1 w} + c_1 e^{r_1 w} + d_1 e^{s_1 w}\big] f^{2 -\frac{p}{2}}-2f (p-1) w_{11}\nonumber\\[2mm]
&- 2(p-1)\big[b_1 e^{q_1 w}+ c_1 e^{r_1 w} + d_1 e^{s_1 w}\big] f^{1 -\frac{p}{2}} w_{11}.
\end{align}

\smallskip
\noindent$\clubsuit$\ We gain the identity  from \eqref{lemma3.1.2}
\begin{equation}\label{lemma3.1.4}
\langle \nabla \Delta_p w, \nabla w \rangle = \frac{p}{2} f^{p/2 - 1} \langle \nabla f, \nabla w \rangle
\ +\ \left[ b_1 q_1 e^{q_1 w} + c_1 r_1 e^{r_1 w} + d_1 s_1 e^{s_1 w} \right] f.
\end{equation}

Substituting \eqref{3.1}-\eqref{3.4} and \eqref{lemma3.1.6}-\eqref{lemma3.1.4} into \eqref{lemma3.1.3}, we obtain
\begin{align}
\frac{f^{2-\alpha-\frac{p}{2}}}{2\alpha} \mathcal{L}(f^\alpha)
&\geq
(2\alpha - 1)(p-1) w_{11}^2
+  \operatorname{Ric}(\nabla w, \nabla w)
+ \frac{1}{n-1}   \left(\sum_{i=2}^n w_{ii}\right)^2
\nonumber\\[2mm]
&\quad + pf w_{11}
+ \big[b_1 q_1 e^{q_1 w} + c_1 r_1 e^{r_1 w} + d_1 s_1 e^{s_1 w}\big] f^{2-\frac{p}{2}}
\nonumber\\[2mm]
&\geq (2\alpha - 1)(p-1) w_{11}^2 +\operatorname{Ric}(\nabla w, \nabla w)+  \frac{f^2}{n-1}  + \left( p - \frac{2(p-1)}{n-1} \right) f w_{11}
 \nonumber\\[2mm]
&\quad+ \frac{\big[b_1 e^{q_1 w} + c_1 e^{r_1 w} + d_1 e^{s_1 w}\big]^2}{n-1} f^{2 - p} + H(w) f^{2 - p/2}
 \nonumber\\[2mm]
&\quad- \frac{2(p-1)\big[b_1 e^{q_1 w} + c_1 e^{r_1 w} + d_1 e^{s_1 w}\big]}{n-1} f^{1 - p/2} w_{11}
\label{lemma3.1.5}
\end{align}
where
\[
H(w) := b_1 \left( q_1 + \frac{2}{n-1} \right) e^{q_1 w} + c_1 \left( r_1 + \frac{2}{n-1} \right) e^{r_1 w} + d_1 \left( s_1 + \frac{2}{n-1} \right) e^{s_1 w}.
\]
Moreover, we notice that
\begin{align*}
&(2\alpha - 1)(p-1) w_{11}^2  - \frac{2(p-1)\big[b_1 e^{q_1 w} + c_1 e^{r_1 w} + d_1 e^{s_1 w}\big]}{n-1} f^{1 - \frac{p}{2}} w_{11}\\[2mm]
&\geq - \frac{(p-1) \big[b_1 e^{q_1 w} + c_1 e^{r_1 w} + d_1 e^{s_1 w}\big]^2 }{(2\alpha - 1) (n-1)} f^{2-p},
\end{align*}
and then find that \eqref{lemma3.1.5} reduces to
\begin{equation}\label{lemma3.1.8}
\begin{aligned}
\frac{f^{2-\alpha-\frac{p}{2}}}{2\alpha} \mathcal{L}(f^\alpha)
&\geq\frac{f^2}{n-1}+B_{n,p,\alpha}  \big[b_1 e^{q_1 w} + c_1 e^{r_1 w} + d_1 e^{s_1 w}\big]^2 f^{2-p}\\
&\quad +\operatorname{Ric}(\nabla w, \nabla w) - \frac{\mathfrak{C}_1}{2}f^{\frac{1}{2}} |\nabla f| + H(w) f^{2 - p/2}
\end{aligned}
\end{equation}
where
$$
\mathfrak{C}_1 := \left| p - \frac{2(p-1)}{n-1} \right|
\qquad \mbox{and}\qquad
B_{n,p,\alpha}  := \frac{1}{n-1} - \frac{p-1}{(2\alpha - 1) (n-1)}\  \overset{\alpha \to \infty}\longrightarrow\ \frac{1}{n-1}>0.
$$

\medskip
\noindent
\textbf{Step 3.  Analysis of $H(w)$.}\\
We will impose some additional assumptions on the range of $p $ to cancel the negativity from $H(w)$.

\noindent {\textbf{Case 1: }}
First of all, we consider the assumption
\begin{equation} \label{condition_q}
\left|\frac{3}{p-1} - \frac{n+1}{n-1}\right| < \frac{2}{n-1},\quad \text{i.e.}  \quad \frac{4n}{n+3} < p < 4.
\end{equation}
Notice that
\begin{align*}
H(w)=\left( q_1 + \frac{2}{n-1} \right)\big[b_1 e^{q_1 w} + c_1 e^{r_1 w} + d_1 e^{s_1 w}\big]+(r_1 - q_1)c_1 e^{r_1 w} + (s_1 - q_1)d_1 e^{s_1 w},
\end{align*}
we deduce the following lower-bound by the basic inequality $A^2 + 2AB \geq -B^2$
\begin{align*}
&B_{n,p,\alpha}  \big[b_1 e^{q_1 w} + c_1 e^{r_1 w} + d_1 e^{s_1 w}\big]^2 f^{2-p} + H(w)f^{2 - p/2}\\[2mm]
\geq &-\dfrac{\left(q_1 + \frac{2}{n-1} \right)^2}{4 B_{n,p,\alpha}} f^{2}+ \left[(r_1 - q_1)c_1 e^{r_1 w} + (s_1 - q_1)d_1 e^{s_1 w}\right]f^{2 - p/2}\\[2mm]
\overset{\eqref{constant1}}{=}&-\dfrac{\left(q_1 + \frac{2}{n-1} \right)^2}{4 B_{n,p,\alpha}} f^{2} + (p-1)^{p-2}e^w \left[-(b+c) u^2+ 2bc u\right]f^{2 - p/2} \geq -\dfrac{\left(q_1 + \frac{2}{n-1}\right)^2}{4 B_{n,p,\alpha}} f^{2}
\end{align*}
where the last inequality has  used the facts that $ b,\, c>0,\: p>1 $ and $0<u <b$.

Consequently, we substitute the estimate above into \eqref{lemma3.1.8} and obtain
\begin{equation*}
\frac{f^{2-\alpha-\frac{p}{2}}}{2\alpha} \mathcal{L}(f^\alpha)\geq\left( \frac{1}{n-1}-\dfrac{\left(q_1 + \frac{2}{n-1} \right)^2}{4 B_{n,p,\alpha}}\right)f^2+\operatorname{Ric}(\nabla w, \nabla w) - \frac{\mathfrak{C}_1}{2}f^{\frac{1}{2}} |\nabla f|.
\end{equation*}
Under the condition $\left|\frac{3}{p-1} - \frac{n+1}{n-1}\right| < \frac{2}{n-1}$, it is trivial that
\begin{align*}
\lim_{\alpha \to +\infty } \beta_{n,p,\alpha} := \lim_{\alpha \to +\infty }\left(\frac{1}{n-1}-\dfrac{\left(q_1 + \frac{2}{n-1} \right)^2}{4 B_{n,p,\alpha}}\right)=\frac{1}{n-1} - \frac{n-1}{4} \left( \frac{3}{p-1} - \frac{n+1}{n-1}\right)^2 >0.
\end{align*}
Thus, there exists a constant $\alpha_1 \geq \dfrac{3}{2}$ in such a way that $\beta_{n,p,\alpha}>0, \,\forall\, \alpha \geq \alpha_1$ and
\begin{equation*}
\frac{f^{2-\alpha-\frac{p}{2}}}{2\alpha} \mathcal{L}(f^\alpha)\geq\beta_{n,p,\alpha}f^2 +\operatorname{Ric}(\nabla w, \nabla w) - \frac{\mathfrak{C}_1}{2}f^{\frac{1}{2}} |\nabla f|.
\end{equation*}

\noindent {\textbf{Case 2: }}
Now we consider the assumption that
\begin{equation} \label{condition_r}
\left|\frac{2}{p-1} - \frac{n+1}{n-1}\right| < \frac{2}{n-1},\quad \text{i.e.}  \quad\frac{3n+1}{n+3} < p < 3.
\end{equation}
We rewrite $H(w)$ by extracting the $r_1$-dominant term as follows
\begin{align*}
H(w) = &\left(r_1 + \frac{2}{n-1}\right)\big[b_1 e^{q_1 w} + c_1 e^{r_1 w} + d_1 e^{s_1 w}\big] + (q_1 - r_1)b_1 e^{q_1 w} + (s_1 - r_1)d_1 e^{s_1 w}.
\end{align*}
As same as before, we obtain a similar estimate by the basic inequality
\begin{align*}
&B_{n,p,\alpha}  \big[b_1 e^{q_1 w} + c_1 e^{r_1 w} + d_1 e^{s_1 w}\big]^2 f^{2-p} + H(w)  f^{2 - p/2}\\[2mm]
\geq&-\dfrac{\left(r_1 + \frac{2}{n-1} \right)^2}{4 B_{n,p,\alpha}} f^{2} + \left[(q_1 - r_1)b_1 e^{q_1 w} + (s_1 - r_1)d_1 e^{s_1 w}\right] f^{2 - p/2}\\[2mm]
\overset{\eqref{constant1}}{=}& -\dfrac{\left(r_1 + \frac{2}{n-1} \right)^2}{4 B_{n,p,\alpha}} f^{2} + (p-1)^{p-2} e^w\left(- u^3 + bc u\right) f^{2 - p/2} \geq -\dfrac{\left(r_1 + \frac{2}{n-1} \right)^2}{4 B_{n,p,\alpha}} f^{2}
\end{align*}
where in the last inequality we have used the fact that
$$
-u^3 + bc u = u(bc - u^2) \geq u(bc - b^2) = ub(c - b) > 0
$$
due to $ c>b>0,\: p>1 $ and $ u = e^{-\frac{w}{p-1}}\in (0, b)$.
As a result, we gain
\begin{equation*}
\frac{f^{2-\alpha-\frac{p}{2}}}{2\alpha} \mathcal{L}(f^\alpha)\geq \left( \frac{1}{n-1}-\dfrac{\left(r_1 + \frac{2}{n-1} \right)^2}{4 B_{n,p,\alpha}}\right)f^2 +\operatorname{Ric}(\nabla w, \nabla w) - \frac{\mathfrak{C}_1}{2}f^{\frac{1}{2}} |\nabla f|.
\end{equation*}
The assumption in \eqref{condition_r} yields
\begin{align*}
\lim_{\alpha \to +\infty } \beta_{n,p,\alpha}:= \lim_{\alpha \to +\infty}\left(\frac{1}{n-1}-\dfrac{\left(r_1 + \frac{2}{n-1} \right)^2}{4 B_{n,p,\alpha}}\right)= \frac{1}{n-1} - \frac{n-1}{4} \left( \frac{2}{p-1} - \frac{n+1}{n-1}\right)^2 >0.
\end{align*}
Thus, there exists a constant $\alpha_2 \geq \dfrac{3}{2}$ such that $\beta_{n,p,\alpha}>0, \, \,\forall\, \alpha \geq \alpha_2$ and
\begin{equation*}
\frac{f^{2-\alpha-\frac{p}{2}}}{2\alpha} \mathcal{L}(f^\alpha)\geq \beta_{n,p,\alpha}f^2 +\operatorname{Ric}(\nabla w, \nabla w) - \frac{\mathfrak{C}_1}{2}f^{\frac{1}{2}} |\nabla f|.
\end{equation*}

\noindent {\textbf{Case 3: }}
In the following, we consider the case that the linear term plays a leading role, and thus it is required to assume the constraint
\begin{equation}\label{condition_s}
\left|\frac{1}{p-1} - \frac{n+1}{n-1}\right| < \frac{2}{n-1},\quad \text{i.e.}  \quad  \frac{2n+2}{n+3} < p < 2.
\end{equation}
The strategy is totally same as in the cases before.
Rewrite $H(w)$ as
\begin{align*}
H(w) =  \left(s_1 + \frac{2}{n-1}\right)\big[b_1 e^{q_1 w} + c_1 e^{r_1 w} + d_1 e^{s_1 w}\big] + (q_1 - s_1)b_1 e^{q_1 w}+(r_1 - s_1)c_1 e^{r_1 w},
\end{align*}
which together with the basic inequality implies
\begin{align*}
&B_{n,p,\alpha}  \big[b_1 e^{q_1 w} + c_1 e^{r_1 w} + d_1 e^{s_1 w}\big]^2 f^{2-p} + H(w)  f^{2 - p/2}\\[2mm]
\geq &-\dfrac{\left(s_1 + \frac{2}{n-1} \right)^2}{4 B_{n,p,\alpha}} f^{2} + \left[(q_1 - s_1)b_1 e^{q_1 w}+(r_1 - s_1)c_1 e^{r_1 w}\right]f^{2 - p/2}\\[2mm]
\overset{\eqref{constant1}}{=}& -\dfrac{\left(s_1 + \frac{2}{n-1} \right)^2}{4 B_{n,p,\alpha}} f^{2}
+(p-1)^{p-2}e^w\left[ -2u^3 + (b+c)u^2\right] f^{2 - p/2}
\geq-\dfrac{\left(s_1 + \frac{2}{n-1} \right)^2}{4 B_{n,p,\alpha}} f^{2}
\end{align*}
where the last inequality has used the fact that
$$
(b+c) - 2u > (b+c) - 2b = c - b > 0, \quad \forall\, u = e^{-\frac{w}{p-1}}\in (0, b).
$$
As a consequence, we still have
\begin{equation*}
\frac{f^{2-\alpha-\frac{p}{2}}}{2\alpha} \mathcal{L}(f^\alpha)\geq\left( \frac{1}{n-1}-\dfrac{\left(s_1 + \frac{2}{n-1} \right)^2}{4 B_{n,p,\alpha}}\right)f^2 +\operatorname{Ric}(\nabla w, \nabla w) - \frac{\mathfrak{C}_1}{2}f^{\frac{1}{2}} |\nabla f|
\end{equation*}
where, under the condition \eqref{condition_s},
\begin{align*}
\lim_{\alpha \to +\infty } \beta_{n,p,\alpha} := \lim_{\alpha \to +\infty }\left(\frac{1}{n-1}-\dfrac{\left(s_1 + \frac{2}{n-1} \right)^2}{4 B_{n,p,\alpha}}\right)=\frac{1}{n-1} - \frac{n-1}{4} \left( \frac{1}{p-1} - \frac{n+1}{n-1}\right)^2 >0.
\end{align*}
Thus, there exists a constant $\alpha_3 \geq \dfrac{3}{2}$ such that $\beta_{n,p,\alpha}>0, \,\forall\, \alpha \geq \alpha_3$ and
\begin{equation*}
\frac{f^{2-\alpha-\frac{p}{2}}}{2\alpha} \mathcal{L}(f^\alpha) \geq \beta_{n,p,\alpha}f^2 +\operatorname{Ric}(\nabla w, \nabla w)- \frac{\mathfrak{C}_1}{2}f^{\frac{1}{2}} |\nabla f|.
\end{equation*}

\noindent{\textbf{Case 4:}}
The admissible range above for  $p$ significantly depending on which of the three nonlinear terms is taken as the principal part. In the present case, we will treat the three terms as a combined function for analysis.

Recall $u= e^{-\frac{w}{p-1}} $ and the definition of $H(w)$, we are enable to obtain
\[
\begin{aligned}
H(w)
= (p-1)^{p-1} u \left[\left(\frac{n+1}{n-1}-\frac{3}{p-1}\right)u^2-(b+c)
\left(\frac{n+1}{n-1}-\frac{2}{p-1}\right)u
+bc\left(\frac{n+1}{n-1}-\frac{1}{p-1}\right)\right].
\end{aligned}
\]
Since $u > 0$ and $(p-1)^{p-1} > 0$, the sign of $H(w)$ is determined by the quadratic
\begin{align*}
Q(u):=\left(\frac{n+1}{n-1}-\frac{3}{p-1}\right)u^2-(b+c)\left(\frac{n+1}{n-1}-\frac{2}{p-1}\right)u
+ bc\left(\frac{n+1}{n-1}-\frac{1}{p-1}\right).
\end{align*}

We need that $Q(u) \geq 0$ for all $u \in (0, b)$, which will imply $ H(w)\geq0$. At the endpoints we have
\[
Q(0)=bc\left(\frac{n+1}{n-1}-\frac{1}{p-1}\right),\qquad
Q(b)=\frac{b(c-b)}{p-1}>0.
\]
Hence $Q(0)\ge0$ together with $b>0, c > 0$ will force
$$\dfrac{n+1}{n-1}-\dfrac{1}{p-1}\ge0,\quad\mbox{i.e.}, \quad p\ge\frac{2n}{n+1}.$$

On the other hand, $p-\frac{4n-2}{n+1}$ and $ \frac{n+1}{n-1} - \frac{3}{p-1}$ have the same sign.
Thus, we need to consider the following three cases.
\smallskip

\noindent$\clubsuit$\
If $$ \dfrac{n+1}{n-1} - \dfrac{3}{p-1} <0,$$
the infimum of $ Q(u)$ on $(0,b)$ occurs at one of the endpoints. Therefore,
$$Q(u) \geq \min\{Q(0), Q(b)\} \geq 0$$
for all $u \in (0, b)$.
\smallskip

\noindent$\clubsuit$\
If  $$\dfrac{n+1}{n-1} - \dfrac{3}{p-1}=0,$$
i.e., the quadratic degenerates to a linear function
$$Q(u) =- \dfrac{b+c}{p-1}u + \dfrac{2bc}{p-1},$$
we observe
$$\inf\limits_{u\in (0,b)} Q(u) = Q(b)>0.$$
\smallskip

\noindent$\clubsuit$\
The nonnegativity of $Q(u)$ with $u \in (0,b)$ fails when
$$ p>\dfrac{4n-2}{n+1},\quad\mbox{i.e.},\quad \dfrac{n+1}{n-1}-\dfrac{3}{p-1}>0.$$
In fact,
\begin{align*}
\min\limits_{u\in (0,b)} Q(u)
&=\dfrac{4bc \left(\dfrac{n+1}{n-1}-\dfrac{1}{p-1}\right)\left(\dfrac{n+1}{n-1}-\dfrac{3}{p-1}\right)- (b+c)^2 \left(\dfrac{n+1}{n-1}-\dfrac{2}{p-1}\right)^2}{4\left(\dfrac{n+1}{n-1}-\dfrac{3}{p-1}\right)} \\
&=\dfrac{4bc \left[\left(\dfrac{n+1}{n-1}-\dfrac{2}{p-1}\right)^2- \left(\dfrac{1}{p-1}\right)^2\right]- (b+c)^2 \left(\dfrac{n+1}{n-1}-\dfrac{2}{p-1}\right)^2}{4\left(\dfrac{n+1}{n-1}-\dfrac{3}{p-1}\right)} \\
&=\dfrac{-(b-c)^2 \left(\dfrac{n+1}{n-1}-\dfrac{2}{p-1}\right)^2- 4bc \left(\dfrac{1}{p-1}\right)^2}{4\left(\dfrac{n+1}{n-1}-\dfrac{3}{p-1}\right)}
<0.
\end{align*}

Therefore, we impose that
\begin{align}
\frac{2n}{n+1} \le p \le \frac{4n-2}{n+1}.
\end{align}
Notice that
$$\lim_{\alpha \to +\infty} B_{n,p,\alpha} = \dfrac{1}{n-1}>0,$$
there exists a constant $\alpha_4 = \alpha_4(n,p) \geq \frac{3}{2}$ such that inequality \eqref{lemma3.1.8} is bounded by
\begin{equation*}
\frac{f^{2-\alpha-\frac{p}{2}}}{2\alpha} \mathcal{L}(f^\alpha)\geq\frac{1}{n-1}f^2+\operatorname{Ric}(\nabla w, \nabla w)- \frac{\mathfrak{C}_1}{2} f^{\frac{1}{2}} |\nabla f|.
\end{equation*}
Thus, we finish the analysis for \textbf{Case 4}.
\medskip

In order to draw a conclusion of \textbf{Step 3}, we denote
\begin{equation*}
\alpha_0 := \alpha_0(n,p)=
\begin{cases}
\alpha_1,\quad \dfrac{4n}{n+3} < p < 4; \\[2mm]
\alpha_2,\quad \dfrac{3n+1}{n+3} < p < 3; \\[2mm]
\alpha_3,\quad \dfrac{2n+2}{n+3} < p < 2;\\[2mm]
\alpha_4,\quad \dfrac{2n}{n+1} \le p \le\dfrac{4n-2}{n+1};
\end{cases}
\end{equation*}
and
\begin{equation}\label{constant3}
\beta_{n,p,\alpha_0} :=
\begin{cases}
\beta_{n,p,\alpha_1}, \quad \dfrac{4n}{n+3} < p < 4; \\[2mm]
\beta_{n,p,\alpha_2},\quad \dfrac{3n+1}{n+3} < p < 3; \\[2mm]
\beta_{n,p,\alpha_3},\quad  \dfrac{2n+2}{n+3} <p < 2; \\[2mm]
\dfrac{1}{n-1},\quad\dfrac{2n}{n+1}\leq p \leq \dfrac{4n-2}{n+1}.
\end{cases}
\end{equation}
It is noted that the above four intervals may intersect, which means the corresponding parameter value, $\alpha_0$ or $\beta_{n,p,\alpha_0}$, is not unique. Therefore, it is specified that the smaller parameter is taken when they intersect. Since
$$\operatorname{Ric}(\nabla w, \nabla w) \geq -(n-1)\kappa f,$$
we summarize all the above results into a unified expression
\[
\mathcal{L}(f^{\alpha_0}) \geq 2\alpha_0 \beta_{n,p,\alpha_0} f^{\alpha_0 + \frac{p}{2}} - 2\alpha_0 (n-1)\kappa f^{\alpha_0 + \frac{p}{2} - 1} - \alpha_0 \mathfrak{C}_1 |\nabla f| f^{\alpha_0 + \frac{p - 3}{2}}.
\]
This completes the proof.
\end{proof}

\smallskip
\begin{remark}\label{remark2.1}
In {\textbf{Step 3}} of the above proof,  the interval $ (0,b)$ for $u$ can be extended to a maximal one in such a way that the analysis still works.
After a direct calculation,  we can replace $(0,\, b)$ by $(0,\, \frac{2bc}{b+c})$ in Case 1, $(0,\, \sqrt{bc})$ in Case 2, $(0, \frac{b+c}{2})$ in Case 3, and
$$
\left(0,\quad \frac{(b+c)\left(\frac{n+1}{n-1}-\frac{2}{p-1}\right) + \sqrt{(b-c)^2 \left(\frac{n+1}{n-1}-\frac{2}{p-1}\right)^2 + \frac{4bc}{(p-1)^2}}}{2\left(\frac{n+1}{n-1}-\frac{3}{p-1}\right)}\, \right)
$$
in Case 4. All of these maximal intervals are the  proper subsets of $(0,c)$. Retrace the proof in above lemma, we claim that $a_1 <v <a_2$ or the maximal intervals above provides the essential boundedness of the solution and is mathematically necessary for the proof. However, it also  reveals that  the logarithmic transform is ill-suited for solutions with the property $ a_1 <v <a_3$, i.e. $ 0< u<c$, since the key inequality $H(w)\ge0$ fails.
\end{remark}

\smallskip
\begin{corollary}\label{corollary2.1}
Suppose $v$ is a solution (as in \autoref{definition}) with the bounds $a_2 <v <a_3$, and
$p \in \left( \frac{2n+2}{n+3},\: 4\right)$,
then there exist two constants $\alpha_0 = \alpha_0(n, p) > \frac{3}{2}$ and $\beta_{n,p,\alpha_0}$ defined in \eqref{constant3} such that the following inequality holds pointwise on $\{x : f(x) > 0\}$:
\begin{align}
\dfrac{f^{2-\alpha_0-\frac{p}{2}}}{2\alpha_0}\mathcal{L}(f^{\alpha_0}) \geq \beta_{n,p,\alpha_0} f^{2} - (n-1)\kappa f - \dfrac{\mathfrak{C}_1}{2} |\nabla f| f^{ \frac{1}{2}},
\end{align}
where $\mathfrak{C}_1 = \left| p - \frac{2(p-1)}{n-1} \right|$, and $f = |\nabla w|^2$ with
\begin{align}
w=-(p-1)\ln \big(a_3-v\big).
\label{transformation-log2}
\end{align}
\end{corollary}

\begin{proof}
Define
\[
\tilde{u}(x) := a_3 - v(x), \qquad
\tilde{c} := a_3 - a_2 > 0, \qquad
\tilde{b} := a_2 - a_1 > 0.
\]
Since $a_2 <v <a_3$ and $\nabla\tilde{u} = -\nabla v$, we have \(\tilde{u} = a_3 - v \in (0, \tilde{c})\),  \(\ |\nabla\tilde{u}| = |\nabla v|\) and therefore
\[
\Delta_p\tilde{u} = \operatorname{div}\bigl(|\nabla\tilde{u}|^{p-2}\nabla\tilde{u}\bigr)
= -\operatorname{div}\bigl(|\nabla v|^{p-2}\nabla v\bigr) = -\Delta_p v.
\]
Substituting \(v = a_3 - \tilde{u}\) into the original nonlinearity gives
\begin{align*}
(v - a_1)(v - a_2)(v - a_3)
= (a_3 - \tilde{u} - a_1)(a_3 - \tilde{u} - a_2)(-\tilde{u})
= -\tilde{u}\bigl(\tilde{u} - \tilde{c}\bigr)\bigl(\tilde{u} - (b + \tilde{c})\bigr),
\end{align*}
where we have used \(\tilde{b} + \tilde{c} = a_3 - a_1\).
Thus the original equation \eqref{GAC} deduces to
\begin{equation*}
\Delta_p\tilde{u} + \tilde{u}\bigl(\tilde{u} - \tilde{c}\bigr)\bigl(\tilde{u} - (\tilde{b} + \tilde{c})\bigr) = 0
\end{equation*}
which  is identical in form to the transformed equation in {\textbf{Step 1}} in the last lemma.
Hence, by applying \autoref{lem:gradient_estimate_bounded}, any result proved for solutions lying in the interval \((a_1, a_2)\) applies directly to \(\tilde{u}\) (and consequently to \(v\)) with the appropriate change of notation. The proof is completed.
\end{proof}

\subsection{ The low-bound under the hyperbolic tangent transformation}\label{section2.2}\

In Section \ref{section2.1}, the  translation transformation together with the logarithmic transformation is effective because it converts algebraic-type nonlinearities involving $v$ into exponential ones, which are convex functions of $w$.
Convexity ensures that linear combinations of different exponential terms, after completing squares in the  estimate of $H(w)$, yield a dominant positive term $\beta_{n,p,\alpha_0} f^2$ under suitable conditions.

Unfortunately, as shown before, the standard logarithmic transformation fails when $v$ occurs from $a_1$,  goes through  $a_2$ and then approaches $a_3$.
To recover the method in Section \ref{section2.1}, for any solution $v$ locating  in $(a_1,a_3)$ and  satisfying \eqref{GAC} we set the conventions
\begin{align}\label{transformHT11111}
 m = \frac{a_1 + a_3}{2}, \quad L = \frac{a_3 - a_1}{2} > 0, \quad \delta = \frac{m - a_2}{L} = \frac{a_1 + a_3 - 2a_2}{a_3 - a_1},
\end{align}
and then introduce the hyperbolic tangent transformation:
\begin{align}\label{transformation-tangent}
v = m + L \tanh w,
\end{align}
which  maps $\mathbb{R}$ onto  the finite interval $(a_1,a_3)$.
By the way, the hyperbolic tangent function is the combination of the exponential function, i.e.,
$$\tanh w = \dfrac{e^w -e^{-w}}{e^w +e^{-w}}.$$
This insight will enable us to obtain the Harnack inequality, see the proof of \autoref{theorem4} for more detail.

\begin{lemma}\label{lemma3.5.1}
Under the transformation \eqref{transformation-tangent}, the $p$-Laplacian equation with  cubic polynomial nonlinearity \eqref{GAC} transforms to
\begin{equation}
\Delta_p w - 2(p-1)\tanh w |\nabla w|^p - L^{4-p}(1 - \tanh^2 w)^{2-p}(\delta + \tanh w) = 0.
\label{eq:transformed}
\end{equation}
\end{lemma}

\begin{proof}
The mapping $v = m + L \tanh w$ is bijective from $w \in (-\infty, \infty)$ to $a_1 <v <a_3$, with parameters $m$ and $L$ as defined in \eqref{transformHT11111}.
Differentiating the transformation $v$, we obtain
\[
\nabla v = L(1 - \tanh^2 w)\nabla w = L\operatorname{sech}^2 w \nabla w\quad\mbox{and}
\quad |\nabla v|^2 = L^2\operatorname{sech}^4 w |\nabla w|^2.
\]
Recall that $\Delta_p v = \operatorname{div}(|\nabla v|^{p-2}\nabla v)$. We compute
\begin{align*}
|\nabla v|^{p-2}\nabla v
= L^{p-2}\operatorname{sech}^{2(p-2)}w \cdot |\nabla w|^{p-2} \cdot (L\operatorname{sech}^2 w \nabla w)
= L^{p-1}\operatorname{sech}^{2p-2}w |\nabla w|^{p-2}\nabla w,
\end{align*}
and thus, by taking the divergence,
\begin{align*}
\Delta_p v
&= L^{p-1}\Big[\nabla(\operatorname{sech}^{2p-2}w) \cdot (|\nabla w|^{p-2}\nabla w) + \operatorname{sech}^{2p-2}w \operatorname{div}(|\nabla w|^{p-2}\nabla w)\Big]
\\[2mm]
&= L^{p-1}\operatorname{sech}^{2p-2}w\Big[\Delta_p w - 2(p-1)\tanh w |\nabla w|^p\Big],
\end{align*}
where we have used
\[
\nabla(\operatorname{sech}^{2p-2}w) = -2(p-1)\tanh w \operatorname{sech}^{2p-2}w \nabla w.
\]
We turn to computing the cubic polynomial term as
\begin{align*}
v - a_1 = L(1 + \tanh w),\quad
v - a_3 = L(\tanh w - 1), \quad
v - a_2 = L(\delta + \tanh w)
\end{align*}
and hence
\begin{align*}
(v-a_1)(v-a_2)(v-a_3) &= L^3(1+\tanh w)(\delta + \tanh w)(\tanh w - 1)
\\[2mm]
&= -L^3\operatorname{sech}^2 w (\delta + \tanh w),
\end{align*}
where we have used
$$(1+\tanh w)(\tanh w - 1) = \tanh^2 w - 1 = -\operatorname{sech}^2 w.$$

Now, substituting the above results into \eqref{GAC} we get
\[
L^{p-1}\operatorname{sech}^{2p-2}w\Big[\Delta_p w - 2(p-1)\tanh w |\nabla w|^p\Big] - L^3\operatorname{sech}^2 w (\delta + \tanh w) = 0
\]
and, by dividing by $L^{p-1}\operatorname{sech}^{2p-2}w$ (which is strictly positive), we then obtain
\[
\Delta_p w - 2(p-1)\tanh w |\nabla w|^p - \frac{L^{4-p}}{\operatorname{sech}^{2p-4}w}(\delta + \tanh w) = 0.
\]
Since
$$\operatorname{sech}^{4-2p}w = (1 - \tanh^2 w)^{2-p},$$
the desired equation can be verified.
\end{proof}

\begin{lemma}\label{lemma3.5.2}
Suppose $ n\geq 2$, and
\begin{align*}
p\in \left[\dfrac{2n}{n+1}, \quad \frac{2n}{n+1} + \frac{(n-1)(1 +\sqrt{1-\delta^2})}{(n+1)\delta^2} \right]\cap\:(1, \infty).
\end{align*}
Then there exists   $\alpha_0 = \alpha_0(n, p) > \frac{3}{2}$ such that for all $ \alpha \geq \alpha_0$ there holds the following inequality  pointwise on $\{x : f(x) > 0\}$:
\begin{align}\label{lowerbound2}
\frac{f^{2-\alpha-\frac{p}{2}}}{2\alpha} \mathcal{L}(f^\alpha) \geq \mathfrak{C_{n,p}}f^2 -(n-1)\kappa f,
\end{align}
where  $f = |\nabla w|^2$ with $w$ given in \eqref{transformation-tangent}, and
\begin{equation}
\mathfrak{C_{n,p}}:=  \begin{cases}
2(p-1), & \text{if}\,\, p \geq n,\\[2mm]
\dfrac{2(p-1)^2}{n-1}, & \text{if}\,\, 1 < p < n.
\end{cases}
\end{equation}
\end{lemma}

\begin{proof}
Recall that the expression of the operator $\mathcal{L}(f^\alpha)$ has been given in \eqref{lemma3.1.3} with $w$ given in \eqref{transformation-tangent}. Its estimate will be carried out in four steps.
\medskip

\noindent\textbf{Step 1. The rough lower-bound.}\\
Let
\begin{align*}
B(w):= L^{4-p}(1-\tanh^2 w)^{2-p}(\delta + \tanh w).
\end{align*}
Differentiating both sides of the equation \eqref{eq:transformed} and taking inner product with $\nabla w$, we arrive at the following
\begin{align}\label{estiamte67}
\begin{aligned}
\langle \nabla \Delta_p w, \nabla w \rangle
&= 2(p-1)\left[ (\nabla \tanh w \cdot \nabla w) f^{p/2} + \tanh w \cdot \frac{p}{2} f^{p/2 - 1} \langle \nabla f, \nabla w \rangle \right] + B'(w) f
\\[2mm]
&= 2(p-1) \operatorname{sech}^2 w f^{1 + p/2} + p(p-1) \tanh w f^{p/2 - 1}\langle \nabla f, \nabla w \rangle + B'(w) f,
\end{aligned}
\end{align}
since $$\nabla \tanh w = (1 - \tanh^2 w) \nabla w = \operatorname{sech}^2 w \nabla w.$$

Choosing an orthonormal frame $\{e_i\}$ such that $e_1 = \frac{\nabla w}{|\nabla w|}$
and substituting  \eqref{3.1}-\eqref{3.4} and \eqref{estiamte67} into \eqref{lemma3.1.3}, we obtain a lower-bound estimate
\begin{align*}
\frac{f^{2-\alpha-\frac{p}{2}}}{2\alpha} \mathcal{L}(f^\alpha)
& \geq
(2\alpha -1) (p-1) w_{11}^2
+ \frac{1}{n-1} \left( \sum_{i=2}^n w_{ii} \right)^2
+ \operatorname{Ric}(\nabla w, \nabla w)
 \\[2mm]
&\quad + f^{1-\frac{p}{2}} \left[ 2(p-1) \operatorname{sech}^2 w f^{1 + p/2} + p(p-1) \tanh w f^{p/2 - 1} \langle \nabla f, \nabla w \rangle + B'(w) f \right]
\\[2mm]
&= (2\alpha -1) (p-1) w_{11}^2
+ \frac{1}{n-1} \left( \sum_{i=2}^n w_{ii} \right)^2
+ \operatorname{Ric}(\nabla w, \nabla w)
 \\[2mm]
 &\quad
 + 2(p-1) \operatorname{sech}^2 w f^2
 + 2p(p-1) w_{11}\tanh w f
 + B'(w) f^{2-\frac{p}{2}}
.
\end{align*}
The $p$-Laplacian can be expressed as
\[
\Delta_p w = f^{p/2 - 1} \left[ (p-1) w_{11} + \sum_{i=2}^n w_{ii} \right].
\]
From equation \eqref{eq:transformed}, we obtain
\[
(p-1) w_{11} + \sum_{i=2}^n w_{ii} = 2(p-1) \tanh w f + B(w) f^{1 - p/2},
\]
and then,
\begin{align*}
\left( \sum_{i=2}^n w_{ii} \right)^2 = & \left[ 2(p-1)\tanh w \cdot f + B(w) f^{1-p/2} - (p-1)w_{11} \right]^2 \\[2mm]
= & 4(p-1)^2\tanh^2 w \cdot f^2 + B(w)^2 f^{2-p} + (p-1)^2 w_{11}^2 \\[2mm]
& + 4(p-1)\tanh w \cdot B(w) f^{2-p/2} - 4(p-1)^2\tanh w \cdot f w_{11}\\
& - 2(p-1) B(w) f^{1-p/2} w_{11}.
\end{align*}
Thus,
\begin{align}\label{eq:lowerbound}
    \begin{split}
\frac{f^{2-\alpha-\frac{p}{2}}}{2\alpha} \mathcal{L}(f^\alpha)
&\geq
\underbrace{(p-1)
\left[
\left( 2\alpha -1  \right) w_{11}^2
- \frac{2B(w)}{n-1} f^{1-p/2} w_{11}
+ 2\tanh w \left( p - \frac{2(p-1)}{n-1} \right) f w_{11}
\right]
}_{:=  \mathcal{T}_1}
 \\[2mm]
&\quad +
 \underbrace{
 \left[ \frac{4(p-1)^2\tanh^2 w}{n-1} + 2(p-1)\operatorname{sech}^2 w \right] f^2
+ \frac{B(w)^2}{n-1} f^{2-p}
}_{:= \mathcal{T}_2}
\\[2mm]
&\quad +
\underbrace{
\left[ \frac{4(p-1)\tanh w \cdot B(w)}{n-1} + B'(w) \right] f^{2-p/2}
 }_{:= \mathcal{T}_3}
\ +\  \operatorname{Ric}(\nabla w, \nabla w).
\end{split}
\end{align}

\medskip
\noindent\textbf{Step 2. Analysis of the nonnegativity of $\mathcal{T}_1 + \mathcal{T}_2$.}\\
 Regarding $ \mathcal{T}_1$ as a  quadratic term of $w_{11}$ and choosing $\alpha > \dfrac{3}{2}$, we observe that the minimum of $\mathcal{T}_1$ is given by
\begin{equation} \label{eq:quadraticbound}
\mathcal{T}_1 \geq
 -\dfrac{p-1}{4\mathcal{A}}\left(\mathcal{B}^2 f^2 + 2\mathcal{B}\mathcal{C} f^{2-p/2} + \mathcal{C}^2 f^{2-p}\right)
 \geq
 -\dfrac{p-1}{4\mathcal{A}}\left(2\mathcal{B}^2 f^2 +  2\mathcal{C}^2 f^{2-p}\right)
,
\end{equation}
where the Young's inequality has been used and the notation is denoted as follows
\[
\mathcal{A} = 2\alpha - 1 > 0, \quad
\mathcal{B}= 2\tanh w\left[p - \frac{2(p-1)}{n-1}\right],
\quad \mathcal{C} = -\frac{2B(w)}{n-1}.
\]
Combining \eqref{eq:quadraticbound} with the term $ \mathcal{T}_2$, we choose $\alpha$ large enough and thus obtain
\begin{align*}
    \mathcal{T}_1 + \mathcal{T}_2 &\geq
     \left[
         \frac{4(p-1)^2\tanh^2 w}{n-1}
         + 2(p-1)\operatorname{sech}^2 w
         - \dfrac{(p-1)\mathcal{B}^2}{2\mathcal{A}}
     \right] f^2
     +
     \left[\frac{B(w)^2}{n-1}
     - \dfrac{(p-1)\mathcal{C}^2}{2\mathcal{A}}\right]f^{2-p}
     \\[2mm]
&\geq
     \left[
         \frac{2(p-1)^2\tanh^2 w}{n-1}
         + 2(p-1)\operatorname{sech}^2 w
     \right] f^2
     +
      \frac{B(w)^2}{2(n-1)}  f^{2-p}
    \\[2mm]
    &\geq
\left[ 2(p-1) + 2(p-1)  \frac{p - n}{n-1} \tanh^2 w\right]f^2
\\[2mm]
&\geq
\begin{cases}
     2(p-1)f^2, \quad & { \rm if} \: p\geq n;
    \\[2mm]
    \dfrac{2(p-1)^2}{n-1} f^2, \quad & { \rm if} \: p<n.
\end{cases}
\end{align*}

\noindent\textbf{Step 3. Analysis of the nonnegativity of $\mathcal{T}_3$.}\\
Recall
$$
\delta  = \dfrac{a_1 + a_3 - 2a_2}{a_3 - a_1}\in (-1,1),
$$
and
$$
B(w) = L^{4-p}(1 - \tanh^2 w)^{2-p}(\delta + \tanh w).
$$
Thus
\begin{align*}
B'(w):= \dfrac{{\rm d}B(w)}{{\rm d}w}
&=L^{4-p}(1 - \tanh^2 w)^{2-p} \left[ (1 - \tanh^2 w) - 2(2-p)\tanh w (\delta + \tanh w) \right]\\[2mm]
&=L^{4-p}(1 - \tanh^2 w)^{2-p} [ (2p-5)\tanh^2 w - 2(2-p)\delta \tanh w +1],
\end{align*}
and
\begin{align*}
\mathcal{T}_3 = L^{4-p}(1- \tanh^2 w)^{2-p}\left[ \frac{2p(n+1)-5n+1}{n-1}\tanh^2 w + \frac{2p(n+1)-4n}{n-1} \delta \tanh w
+ 1 \right]f^{2-p/2},
\end{align*}
where $ L>0$ and
$$1- \tanh^2 w=\operatorname{sech}^2 w \in (0,1].$$

Define
\begin{align}
\mathfrak{g}( z ) =   \frac{2p(n+1)-5n+1}{n-1}z^2 + \frac{2p(n+1)-4n}{n-1} \delta z + 1, \quad \forall\, z\in {\mathbb R}.
\label{function-g}
\end{align}
We will analyze the infimum of $\mathfrak{g}( \tanh w )$ with $ \tanh w  \in (-1,1)$ and then confirm $\mathfrak{g}( \tanh w ) \geq 0 $ so that $\mathcal{T}_3\geq 0$.
This can be done in the following way.

\smallskip\noindent
\textbf{Case 1:} Arguments on the case {$2p(n+1) - 5n + 1 = 0$, i.e., $p = \frac{5n-1}{2(n+1)}$.} \\
Since $ \delta\in (-1,1)$ and $\tanh w\in (-1,1)$, there holds
\[
\mathfrak{g}( \tanh w)=\delta  \tanh w + 1>0.
\]
\smallskip

\noindent\textbf{Case 2:} Arguments on the case $2p(n+1) - 5n + 1> 0$, i.e., $p > \frac{5n-1}{2(n+1)}$. \\
The function $\mathfrak{g}( z)$ on ${\mathbb R}$ is a convex quadratic function  with a minimum at
\[
z_0
= -\frac{\delta[2p(n+1)-4n]}{2[2p(n+1) - 5n + 1]}.
\]
The infimum value of $\mathfrak{g}(\tanh w)$ with $\tanh w\in (-1,1)$ is either at $z_0$ (if $z_0 \in (-1,1)$) or at one of the endpoints.
\smallskip

\noindent $\clubsuit$\
If $\delta=0$, then $ z_0 =0$. Hence, $\mathfrak{g}( \tanh w) \geq 1$ due to \eqref{function-g}.

\smallskip\noindent
$\clubsuit$\
We assume $\delta \neq 0$ and, further,
$$ p \leq  \dfrac{2n}{n+1} + \dfrac{n-1}{(n+1)(2 - |\delta|)}.$$
This implies  $z_0 \in \mathbb{R}\backslash(-1,1)$.
If $z_0 \le -1$,  the infimum on $(-1,1)$ occurs at $t = -1$:
\[
\mathfrak{g}(-1) = \frac{2p(n+1) - 5n + 1}{n-1} - \frac{[2p(n+1)-4n]\delta}{n-1} + 1
=
\frac{[2p(n+1)-4n]}{n-1}(1-\delta)>0.
\]
If $z_0 \ge 1$, the infimum on $(-1, 1)$ occurs at $t = 1$:
\[
\mathfrak{g}(1) = \frac{2p(n+1) - 5n + 1}{n-1} + \frac{[2p(n+1)-4n]\delta}{n-1} + 1
=\frac{[2p(n+1)-4n]}{n-1}(1+\delta)>0.
\]
Summing up the permissible range of $p$ above, we obtain
\begin{align*}
p \in\left(
\dfrac{5n-1}{2(n+1)},\quad \dfrac{2n}{n+1} + \dfrac{n-1}{(n+1)(2 - |\delta|)}
\right]
\end{align*}
where the assumption $\delta \neq 0$ makes sure the upper-bound of the set belongs to $\left( \frac{5n-1}{2(n+1)},\: \frac{3n-1}{n+1}\right)$ and thus the range is nonempty.
\smallskip

\noindent $\clubsuit$\
We assume $\delta \neq 0$ and
$$ p > \dfrac{2n}{n+1} + \dfrac{n-1}{(n+1)(2 - |\delta|)}$$
which implies  $z_0 \in (-1,1)$, we observe
\[
 \mathfrak{g}(z_0)  = 1 - \frac{ \delta^2\left( \frac{2p(n+1)-4n}{n-1} \right)^2}{4\left( \frac{2p(n+1) - 5n + 1}{n-1} \right)}.
\]
We require the additional condition to confirm  $ \mathfrak{g}(z_0) \geq 0$, that is
\begin{align*}
\frac{2n}{n+1} + \frac{(n-1)(1 -\sqrt{1-\delta^2})}{(n+1)\delta^2}\le p \le\frac{2n}{n+1} + \frac{(n-1)(1 +\sqrt{1-\delta^2})}{(n+1)\delta^2}.
\end{align*}
A direct calculation shows, for any $\delta \in (-1,1)\backslash\{0\}$ there holds
\begin{align*}
\dfrac{5n-1}{2(n+1)} <& \frac{2n}{n+1} + \frac{(n-1)(1 -\sqrt{1-\delta^2})}{(n+1)\delta^2}\\
<& \dfrac{2n}{n+1} + \dfrac{n-1}{(n+1)(2 - |\delta|)}\\
<& \frac{2n}{n+1} + \frac{(n-1)(1 +\sqrt{1-\delta^2})}{(n+1)\delta^2}.
\end{align*}
We thus summarize the range of $p$ referring to  $ \delta \in (-1,1)$ as follows:
\begin{align*}
p\in \left(\dfrac{5n-1}{2(n+1)}, \quad \dfrac{2n}{n+1} + \frac{(n-1)\big(1 +\sqrt{1-\delta^2}\,\big)}{(n+1)\delta^2} \right].
\end{align*}
\smallskip

\noindent\textbf{Case 3:} Arguments on the case $2p(n+1) - 5n + 1 < 0$, i.e., $p < \frac{5n-1}{2(n+1)}$. \\
The infimum of $\mathfrak{g}(\tanh w)$ with $\tanh w\in (-1,1)$ must occur at one of the endpoints:
\[
\mathfrak{g}(-1)  = \frac{2p(n+1) - 5n + 1}{n-1} - \frac{[2p(n+1)-4n]\delta}{n-1} + 1,
\]

\[
\mathfrak{g}(1)  = \frac{2p(n+1) - 5n + 1}{n-1} + \frac{[2p(n+1)-4n]\delta}{n-1} + 1.
\]
Since $ \delta \in (-1,1)$, it yields
\begin{align*}
\mathfrak{g}(-1)\geq 0 \Leftrightarrow(2p(n+1)  - 4n)(1 + \delta) \ge 0 \quad \Leftrightarrow\quad p \ge \frac{2n}{n+1},\\[2mm]
\mathfrak{g}(1)\geq 0 \Leftrightarrow (2p(n+1)- 4n)(1 -\delta) \ge 0 \quad \Leftrightarrow\quad p \ge \dfrac{2n}{n+1}.
\end{align*}
We have  both $\mathfrak{g}(-1) \ge 0$ and $\mathfrak{g}(1) \ge 0$ for
\begin{align*}
p\in \left[\dfrac{2n}{n+1},\: \dfrac{5n-1}{2(n+1)}\right)\neq \varnothing
\end{align*}
since $\dfrac{5n-1}{2(n+1)} -\dfrac{2n}{n+1} = \dfrac{n-1}{2(n+1)} > 0$ for $n\ge2$.

Up to now, we have finished the analysis of all the situations.
\medskip

\noindent
\textbf{Step 4. The permissible range of $p$.}\\
Taking account to the range of the $p$-Laplacian, i.e., $ p \in (1, \infty)$,
we claim that the permissible range of $p $ for the nonnegativity of $\mathcal{T}_3$ is
\begin{align*}
p\in \left[\dfrac{2n}{n+1}, \quad \frac{2n}{n+1} + \frac{(n-1)(1 +\sqrt{1-\delta^2})}{(n+1)\delta^2}\right]\cap\:(1, \infty)\quad {\rm for }\quad  \delta \in (-1, 1).
\end{align*}

Combining all these cases and replacing $   \operatorname{Ric}(\nabla w, \nabla w)$ by its lower bound $-(n-1)\kappa f$ in \eqref{eq:lowerbound}, we claim that
\begin{align*}
    \frac{f^{2-\alpha-\frac{p}{2}}}{2\alpha} \mathcal{L}(f^\alpha)
    \geq
    -(n-1)\kappa f+
     \begin{cases}
        2(p-1)f^2, & \text{if } p \geq n,
        \\[2mm]
        \dfrac{2(p-1)^2}{n-1}f^2, & \text{if } 1 < p < n.
      \end{cases}
\end{align*}
The proof is completed.
\end{proof}

\section{Proofs of  main theorems}
Recall that: if $v$ is a solution to \eqref{GAC}, then $f=|\nabla w|^2$ and $w$ is given by \eqref{transformation-log}, or \eqref{transformation-log2}, or \eqref{transformation-tangent}.
For the convenience of notation, we may use these conventions without any further notification throughout this section.

\subsection{ The integral inequality  to  $f=|\nabla w|^2$}\label{sec2.2}\

In the following, a vital integral estimate to $f=|\nabla w|^2$ will be proved, which will play an essential role in the Moser iteration.
Comparing the two pointwise lower bounds \eqref{lowerbound1} and \eqref{lowerbound2},
we will begin with the former and only present the detailed proof of the corresponding integral estimate, because it is a little bit more complicated.

For the purpose of doing that, we first recall the Saloff-Coste's Sobolev inequality.
\begin{lemma}[\cite{MR1180389}]\label{SC-Sobolev}
Let $(M^{n},g)$ be a complete Riemannian manifold with $\mathrm{Ric}_{g}\geq-(n-1)\kappa$ where $\kappa$ is a nonnegative constant.
For $n\geq 2$, there exists a positive constant $C_{n}$ depending only on $n$, such that for all ball $B\subset M$ of radius $R$ and volume $V$ we have
\[
\forall\, f\in C^{\infty}_{0}(B),\quad \|f\|^{2}_{L^{\frac{2n}{n-2}}(B)}\leq e^{C_{n}\big(1+\sqrt{\kappa}R\big)}V^{-\frac{2}{n}}R^{2}\left(\int\limits_{B}|\nabla f|^{2}+R^{-2}f^{2}\right).
\]
For $n=2$, the above inequality holds if we replace $n$ by some fixed $n^{\prime}>2$.
\qed
\end{lemma}

\begin{lemma}\label{lemma2}
    Suppose that all assumptions in   \autoref{theorem1} are true, and $\Omega = B_R (o) \subset M$ is a geodesic ball, then it holds the integral inequality
\begin{align*}
        &\quad
       \beta_{n,p,\alpha_0} \int_\Omega f^{\alpha_0  + \frac{p}{2} +t} \eta^2
        \ +\ \frac{\mathfrak{C}_3}{t} e^{-t_0} V^{\frac{2}{n}} R^{-2}
            \left\|
                f^{\frac{\alpha_0 +t -1 }{2}+ \frac{p}{4}} \eta
            \right\|_{L^\frac{2n}{n-2}}^2
            \\[2mm]
            &\leq
            \mathfrak{C}_5 t_0^2 R^{-2} \int_\Omega f^{\alpha_0 +t+\frac{p}{2}-1} \eta^2
            \ +\ \frac{\mathfrak{C}_4}{t}  \int_\Omega f^{\alpha_0 +t+\frac{p}{2}-1} |\nabla \eta|^2,
    \end{align*}
    where $\beta_{n,p,\alpha_0}$ and $\alpha_0 $ have been chosen in  \eqref{constant3} and the constants $\mathfrak{C}_1, \mathfrak{C}_2, \cdots, \mathfrak{C}_{11}$ depend on $ n,p$.
\end{lemma}
\begin{proof}
We choose a test function $\psi = (f_\varepsilon)^t\, \eta^2$, where $\eta \in C_0^\infty (\Omega, \mathbb{R})$ is a non-negative cut-off function,
$f_\varepsilon = (f-\varepsilon)^+ $  with $\varepsilon >0$, the constant $t>1 $ is to be determined later.
Then, multiplying  \eqref{lowerbound1} by the test function and integrating over $\Omega$, we obtain
\begin{align*}
        &\quad 2\beta_{n,p,\alpha_0} \alpha_0 \int_\Omega f^{\alpha_0 + \frac{p}{2}}
         f_\varepsilon^{t} \eta^2
         -    2\alpha_0 (n-1) \kappa \int_{\Omega}
 f^{\alpha_0 + \frac{p}{2}-1} f_\varepsilon^{t} \eta^2
       -
    \mathfrak{C}_1 \alpha_0 \int_{\Omega}f^{\alpha_0 + \frac{p}{2}-\frac{3}{2}} |\nabla f|^2 f_\varepsilon^{t} \eta^2 \notag
    \\[2mm] \notag
     &\leq
        -\int_\Omega
        \left\langle
                f^{\frac{p}{2}-1} \nabla f^{\alpha_0}
                + (p-2) f^{\frac{p}{2}-2}
                    \left\langle
                        \nabla f^{\alpha_0}, \nabla u
                    \right\rangle\nabla u ,\ \nabla \psi
        \right\rangle
        \\[2mm]\notag
        &=
        -\int_\Omega t \alpha_0  f^{\alpha_0 + \frac{p}{2} -2} f_\varepsilon^{t-1} |\nabla f|^2 \eta^2
        + (p-2) t \alpha_0 f^{\alpha_0 + \frac{p}{2} -3} f_\varepsilon^{t-1} \langle \nabla f, \nabla u\rangle^2 \eta^2
        \\[2mm]
        &\quad
        - \int_\Omega  2\eta \alpha_0 f^{\alpha_0 + \frac{p}{2} -2} f_\varepsilon^{t}
            \langle \nabla f, \nabla \eta \rangle
        +  2\eta \alpha_0   (p-2) f^{\alpha_0 + \frac{p}{2} -3 } f_\varepsilon^{t}
            \langle \nabla f, \nabla u\rangle \langle \nabla u, \nabla \eta \rangle.
     \end{align*}
Using the inequalities that
\begin{align*}
    f^{t-1}_\varepsilon |\nabla f|^2 + (p-2) f_\varepsilon^{t-1} f^{-1} \langle \nabla f, \nabla u\rangle^2 \geq  \mathfrak{C}_2 f_\varepsilon^{t-1} |\nabla f|^2
\end{align*}
where  $ \mathfrak{C}_2=\mathrm{min}\{ 1,p-1\}>\frac{p-1}{p}$, and
\begin{align*}
     f_\varepsilon^{t}
           \langle \nabla f, \nabla \eta \rangle
           + (p-2)  f_\varepsilon^{t}  f^{-1}
            \langle \nabla f, \nabla u\rangle \langle \nabla u, \nabla \eta \rangle
            \geq
            -(p+1)  f_\varepsilon^{t} | \nabla f | | \nabla \eta |,
\end{align*}
and then letting $\varepsilon \rightarrow 0$, we obtain
\begin{equation*}
\begin{split}
        &\quad 2\beta_{n,p,\alpha_0}  \int_\Omega f^{\alpha_0 + \frac{p}{2}+t} \eta^2
        + \mathfrak{C}_2 t  \int_\Omega  f^{\alpha_0 + \frac{p}{2}+t-3} | \nabla f |^2  \eta^2
        \\[2mm]
        &\leq
          2 (n-1) \kappa \int_{\Omega}  f^{\alpha_0 + \frac{p}{2}+t-1} \eta^2
          + \mathfrak{C}_1 \int_\Omega f^{\alpha_0 +\frac{p-3}{2} +t} |\nabla f|^2 \eta^2
          + 2(p+1) \int_\Omega f^{\alpha_0 +\frac{p}{2} +t-2} |\nabla f| |\nabla \eta| \eta.
    \end{split}
\end{equation*}
We have to claim that these integrals make sense since $u \in W_{loc}^{2,2} (\Omega) \cap C^{1,\beta}(\Omega) $, which deduced $f \in C^{\beta}(\Omega)$ and $|\nabla f| \in L^2_{loc}$ (cf. \cite{MR709038, MR727034, MR474389}).

Notice the following two Cauchy's inequalities
\begin{align*}
    \mathfrak{C}_1 f^{\alpha_0 +\frac{p-3}{2} +t} |\nabla f|\eta^2
    \leq
    \frac{\mathfrak{C}_2 t}{4} f^{\alpha_0 +\frac{p}{2} +t-3}  |\nabla f|^2\eta^2
        + \frac{\mathfrak{C}_1^2}{\mathfrak{C}_2 t } f^{\alpha_0 +\frac{p}{2} +t }   \eta^2,
\end{align*}
and
\begin{align*}
        2(p+1) f^{\alpha_0 +\frac{p}{2} +t-2} |\nabla f| |\nabla \eta| \eta
        \leq
         \frac{\mathfrak{C}_2 t}{4} f^{\alpha_0 +\frac{p}{2} +t-3}  |\nabla f|^2\eta^2
        \ +\
         \frac{4(p+1)^2}{\mathfrak{C}_2t} f^{\alpha_0 +\frac{p}{2} +t-1}  |\nabla \eta|^2,
\end{align*}
then it deduces that,
\begin{equation}
          \label{3.21-1111}
\begin{aligned}
        &\beta_{n,p,\alpha_0}  \int_\Omega f^{\alpha_0 + \frac{p}{2}+t} \eta^2
        \ +\ \frac{\mathfrak{C}_2 t }{2} \int_\Omega  f^{\alpha_0 + \frac{p}{2}+t-3} | \nabla f |^2  \eta^2
        \\[2mm]
        &\leq
         2 (n-1) \kappa \int_{\Omega}  f^{\alpha_0 + \frac{p}{2}+t-1} \eta^2
          \ +\      \frac{4(p+1)^2}{\mathfrak{C}_2t} \int_\Omega f^{\alpha_0 +\frac{p}{2} +t-1}  |\nabla \eta|^2,
\end{aligned}
\end{equation}
where  we  have chosen $t>1$ large enough such that
\begin{align}\label{3.21}
     \frac{\mathfrak{C}_1^2 }{\mathfrak{C}_2 t} \leq\beta_{n,p,\alpha_0}.
\end{align}

Moreover, substituting the following estimate into the left hand side of \eqref{3.21-1111}
\begin{align*}
    \frac{1}{2}
    \left|
        \nabla
            \left(
                        f^{\frac{\alpha_0 +t-1}{2} +\frac{p}{4}} \eta
            \right)
    \right|^2
&\leq
  \left|
        \nabla    f^{\frac{\alpha_0 +t-1}{2} +\frac{p}{4}}
        \right|^2
       \ +\    f^{\alpha_0 +t-1+\frac{p}{2}} | \nabla \eta|^2
        \\[2mm]
        &= \frac{(2\alpha_0 +2t+p -2)^2}{16} f^{\alpha_0 +t+\frac{p}{2}-3} | \nabla f|^2 \eta^2
           \ +\ f^{\alpha_0 +t-1+\frac{p}{2}}| \nabla \eta|^2,
\end{align*}
it arrives at
\begin{align*}
&\quad\beta_{n,p,\alpha_0}  \int_\Omega f^{\alpha_0 + \frac{p}{2}+t} \eta^2 \ +\ \frac{4\mathfrak{C}_2 t }{(2\alpha_0 +2t+p -2)^2} \int_\Omega
\left|\nabla\left(f^{\frac{\alpha_0 +t-1}{2} +\frac{p}{4}} \eta\right)\right|^2\\[2mm]
&\leq 2(n-1)\kappa\int_{\Omega} f^{\alpha_0 + \frac{p}{2}+t-1}\eta^2 \ +\  \left(\frac{4(p+1)^2}{\mathfrak{C}_2t} + \frac{8\mathfrak{C}_2 t }{(2\alpha_0 +2t+p -2)^2}\right)\int_\Omega f^{\alpha_0 +\frac{p}{2} +t-1}|\nabla \eta|^2.
\end{align*}
We claim that there exist $\mathfrak{C}_3= \frac{4\mathfrak{C}_2}{(2\alpha_0 + p)^2}$ and $ \mathfrak{C}_4= \frac{4(p+1)^2}{\mathfrak{C}_2}+ 2\mathfrak{C}_2  $  satisfying
\begin{align*}
\frac{\mathfrak{C}_3}{t}\leq  \frac{4\mathfrak{C}_2 t }{(2\alpha_0 +2t+p -2)^2}
\quad\quad {\rm and }\quad\quad \frac{4(p+1)^2}{\mathfrak{C}_2t} + \frac{8\mathfrak{C}_2 t}{(2\alpha_0 +2t+p -2)^2} \leq \frac{\mathfrak{C}_4}{t}.
\end{align*}
Hence,
\begin{align*}
&\quad\beta_{n,p,\alpha_0}  \int_\Omega f^{\alpha_0 + \frac{p}{2}+t} \eta^2 + \frac{\mathfrak{C}_3 }{t} \int_\Omega\left|\nabla\left(f^{\frac{\alpha_0 +t-1}{2} +\frac{p}{4}} \eta\right)\right|^2\\
\leq& 2(n-1)\kappa\int_{\Omega}f^{\alpha_0 + \frac{p}{2}+t-1} \eta^2 + \frac{\mathfrak{C}_4 }{t}\int_\Omega f^{\alpha_0 +\frac{p}{2} +t-1}|\nabla \eta|^2.
\end{align*}

After that, due to the Saloff-Coste's Sobolev inequality in \autoref{SC-Sobolev}
\begin{align*}
&e^{-C_n(1+ \sqrt{\kappa}R)} V^{\frac{2}{n}} R^{-2}\left\|f^{\frac{\alpha_0 + t-1}{2}+ \frac{p}{4}} \eta\right\|^2_{L^\frac{2n}{n-2}(\Omega)} \\
\leq& \int_\Omega \left| \nabla \left( f^{\frac{\alpha_0 + t-1}{2}+ \frac{p}{4}}\eta \right)\right|^2 + R^{-2} \int_{\Omega} f^{\alpha_0 +t +\frac{p}{2}-1}\eta^2,\quad {\rm for} \: n\geq 3,
\end{align*}
and
\begin{align*}
&e^{-C_{n^{\prime}}(1+ \sqrt{\kappa}R)} V^\frac{1}{2} R^{-2}\left\|f^{\frac{\alpha_0 + t-1}{2}+ \frac{p}{4}} \eta\right\|^2_{L^4(\Omega)} \\
\leq& \int_\Omega \left|\nabla \left( f^{\frac{\alpha_0 + t-1}{2}+\frac{p}{4}}\eta\right)\right|^2 + R^{-2} \int_{\Omega} f^{\alpha_0 +t +\frac{p}{2}-1} \eta^2,\quad {\rm for} \: n=2, \: n^{\prime}=4,
\end{align*}
we obtain
\begin{align*}
&\quad\beta_{n,p,\alpha_0}  \int_\Omega f^{\alpha_0 + \frac{p}{2}+t} \eta^2 + \frac{\mathfrak{C}_3 }{t}e^{-C_n(1+ \sqrt{\kappa}R)} V^{\frac{2}{n}} R^{-2}\left\|f^{\frac{\alpha_0 + t-1}{2}+ \frac{p}{4}} \eta\right\|^2_{L^\frac{2n}{n-2}(\Omega)}\\[2mm]
&\leq 2(n-1)\kappa \int_{\Omega}  f^{\alpha_0 + \frac{p}{2}+t-1} \eta^2 + \frac{\mathfrak{C}_4 }{t}\int_\Omega f^{\alpha_0 +\frac{p}{2} +t-1}  |\nabla \eta|^2 + \frac{\mathfrak{C}_3}{tR^{2}}\int_{\Omega} f^{\alpha_0 +t +\frac{p}{2}-1}\eta^2,\quad  { \rm for} \:n\geq 3,
\end{align*}
and
\begin{align*}
&\quad \beta_{2,p,\alpha_0}  \int_\Omega f^{\alpha_0 + \frac{p}{2}+t} \eta^2 + \frac{\mathfrak{C}_3 }{t} e^{-C_4(1+ \sqrt{\kappa}R)} V^\frac{1}{2} R^{-2}\left\|f^{\frac{\alpha_0 + t-1}{2}+ \frac{p}{4}}\eta\right\|^2_{L^4(\Omega)}\\[2mm]
&\leq 2\kappa \int_{\Omega}  f^{\alpha_0 + \frac{p}{2}+t-1} \eta^2 + \frac{\mathfrak{C}_4 }{t}\int_\Omega f^{\alpha_0 +\frac{p}{2} +t-1}|\nabla \eta|^2 + \frac{\mathfrak{C}_3}{tR^{2}}\int_{\Omega} f^{\alpha_0 +t +\frac{p}{2}-1} \eta^2, \quad  {\rm for} \: n=2.
\end{align*}
Here and after, for convenience, we fix $n^{\prime}=4$ for the case $n=2$. Setting
\begin{align*}
c_1(n,p,\alpha_0) = \max \left\{  C_n+1,\, 2,\, \frac{\mathfrak{C}_1^2}{\mathfrak{C}_2\beta_{n,p,\alpha_0}}\right\}
\end{align*}
and $t_0 = c_1(n,p,\alpha_0)(1+ \sqrt{\kappa}R),$ then for any $t\geq t_0$, the inequality \eqref{3.21} is true. Therefore, by choosing $t\geq t_0$ satisfying
\begin{align*}
2(n-1) \kappa R^2 \leq\frac{2(n-1)}{c_1^2 (n,p,\alpha_0)} t_0^2 \quad\quad \mathrm{and} \quad\quad
\frac{\mathfrak{C}_3}{t} \leq \frac{\mathfrak{C}_3}{c_1 (n,p,\alpha_0)},
\end{align*}
there exists a non-negative constant $\mathfrak{C}_5 = \mathfrak{C}_5(n,p)$ such that
\begin{align*}
2(n-1) \kappa R^2 + \frac{\mathfrak{C}_3}{t}  \leq \mathfrak{C}_5 t_0^2.
\end{align*}
Thus,
\begin{equation}\label{3.27}
\begin{split}
 &\frac{\mathfrak{C}_5 t_0^2 }{R^2}  \int_{\Omega}  f^{\alpha_0 + \frac{p}{2}+t-1} \eta^2
          + \frac{\mathfrak{C}_4 }{t}
                \int_\Omega f^{\alpha_0 +\frac{p}{2} +t-1}  |\nabla \eta|^2
                \\[2mm]
                \geq &\left\{
\begin{aligned}
    \beta_{n,p,\alpha_0}  \int_\Omega f^{\alpha_0 + \frac{p}{2}+t} \eta^2
    \ +\ \frac{\mathfrak{C}_3 }{t}    e^{-t_0} V^{\frac{2}{n}} R^{-2}
    \left\|
            f^{\frac{\alpha_0 + t-1}{2}+ \frac{p}{4}} \eta
    \right\|^2_{L^\frac{2n}{n-2}(\Omega)}, \quad {\rm for} \:n\geq 3,
    \\[2mm]
    \beta_{2,p, \alpha_0}  \int_\Omega f^{\alpha_0 + \frac{p}{2}+t} \eta^2
     \ +\
      \frac{\mathfrak{C}_3 }{t}    e^{-t_0}  V^\frac{1}{2} R^{-2}
    \left\|
            f^{\frac{\alpha_0 + t-1}{2}+ \frac{p}{4}} \eta
    \right\|^2_{L^4(\Omega)}, \quad {\rm for} \: n=2.
\end{aligned}
    \right.
    \end{split}
    \end{equation}
The proof is completed.
\end{proof}

\subsection{ Local $L^\gamma$-upper bound  of the gradient }\label{sec2.3}\

In this section, we will prove the $L^\gamma$-upper bound of the gradient of the weak solutions to the $p $-Laplacian equation with  cubic polynomial nonlinearity \eqref{GAC}.
\begin{lemma}\label{lemma3}
Let $(M,g) $ be a complete Riemannian manifold with $\mathrm{Ric} \geq -(n-1) \kappa $. Suppose that $v $ is a weak solutions to the $p $-Laplacian equation with  cubic polynomial nonlinearity \eqref{GAC} on the geodesic ball $B_R(o) \subset M $, then there exists a constant $\mathfrak{C}_8 = \mathfrak{C}_8(n,p) >0$ such that
\begin{align}\label{3.28}
\|f\|_{L^\gamma( B_{\frac{3R}{4}}(o))}\leq\mathfrak{C}_8 V^{\frac{1}{\gamma}}\frac{t_0^2}{R^2},\quad \gamma :=\left\{
\begin{aligned}
\frac{n}{n-2} \left( \alpha_0 + \frac{p}{2}+ t_0  -1 \right), &\quad {\rm for } \:  n \geq 3, \\[2mm]
2\left( \alpha_0 + \frac{p}{2}+ t_0  -1 \right) , &\quad {\rm for } \: n =2,
\end{aligned}
\right.
\end{align}
where $V$ is the volume of geodesic ball $B_R(o),$ and $\alpha_0$ is fixed as same as in \autoref{lemma2}.
\end{lemma}

\begin{proof}
Setting
$$\Omega_1 = \left\{f> \frac{2\mathfrak{C}_5 t_0^2}{\beta_{n,p,\alpha_0} R^2}\right\}\quad\mbox{and}\quad \Omega_2 = \Omega  \setminus \Omega_1,$$
a direct calculation shows
\begin{align*}
\frac{ \mathfrak{C}_5t_0^2}{R^2} \int_\Omega f^{\alpha_0 + \frac{p}{2}+ t  -1}\eta^2
&=\frac{ \mathfrak{C}_5t_0^2}{R^2} \int_{\Omega_1} f^{\alpha_0 + \frac{p}{2}+ t  -1}\eta^2\ +\ \frac{ \mathfrak{C}_5t_0^2}{R^2} \int_{\Omega_2}  f^{\alpha_0 + \frac{p}{2}+ t  -1}\eta^2\\[2mm]
& \leq \frac{\beta_{n,p,\alpha_0}}{2}\int_\Omega f^{\alpha_0 + \frac{p}{2}+ t }\eta^2 \ +\ \frac{2\mathfrak{C}_5 t_0^2 V}{R^2} \left( \frac{2\mathfrak{C}_5 t_0^2}{\beta_{n,p,\alpha_0}R^2}\right)^{\alpha_0 + \frac{p}{2}+ t  -1}
\end{align*}
where the volume of $B_R(o)$ denoted by $V.$  Recall \eqref{3.27}, then it draws, by substituting above inequality and choosing $t=t_0$,
\begin{equation}\label{3.30}
\begin{split}
          &\quad
          \frac{\beta_{n,p,\alpha_0}}{2}  \int_\Omega f^{\alpha_0 + \frac{p}{2}+t} \eta^2
    + \frac{\mathfrak{C}_3 }{t_0}    e^{-t_0} V^{\frac{2}{n}} R^{-2}
    \left\|
            f^{\frac{\alpha_0 + t-1}{2}+ \frac{p}{4}} \eta
    \right\|^2_{L^\frac{2n}{n-2}(\Omega)}
    \\[2mm]
    & \leq
        \frac{2\mathfrak{C}_5 t_0^2 V}{R^2}
        \left(
                      \frac{2\mathfrak{C}_5 t_0^2}{\beta_{n,p,\alpha_0}R^2}
        \right)^{\alpha_0 + t_0 + \frac{p}{2}-1}
        + \frac{\mathfrak{C}_4 }{t_0}
                \int_\Omega f^{\alpha_0 +\frac{p}{2} +t-1}  |\nabla \eta|^2, \quad { \rm for} \:n\geq 3,
    \end{split}
\end{equation}
and
\begin{equation}\label{3.87}
\begin{split}
 &\quad \beta_{2,p,\alpha_0}  \int_\Omega f^{\alpha_0 + \frac{p}{2}+t} \eta^2
      +
      \frac{\mathfrak{C}_3}{t_0}    e^{-t_0}  V^\frac{1}{2} R^{-2}
    \left\|
            f^{\frac{\alpha_0 + t-1}{2}+ \frac{p}{4}} \eta
    \right\|^2_{L^4(\Omega)}
   \\[2mm]
    & \leq
        \frac{2\mathfrak{C}_5 t_0^2 V}{R^2}
        \left(
                      \frac{2\mathfrak{C}_5 t_0^2}{\beta_{2,p,\alpha_0}R^2}
        \right)^{\alpha_0 + t_0 + \frac{p}{2}-1}
        + \frac{\mathfrak{C}_4 }{t_0}
                \int_\Omega f^{\alpha_0 +\frac{p}{2} +t-1}  |\nabla \eta|^2,\quad { \rm for} \:n=2.
\end{split}
\end{equation}
According to the method of cut-off function, we choose the function $\eta_1 \in C_0^{\infty } (B_R(o)) $ satisfying
\begin{equation*}
0\leq \eta_1 \leq 1,
\qquad
|\nabla \eta_1 | \leq \frac{C(n)}{R} \quad  \mathrm{in}  \  B_{R} (o),
\qquad
\eta_1 \equiv 1 \quad \mathrm{in}  \  B_{\frac{3R}{4}} (o).
\end{equation*}
As a consequence, by considering the proposition of above cut-off function, choosing $\eta = \eta_1^{\alpha_0 + \frac{p}{2}+ t_0  }$, and using the H\"older's inequality and the Cauchy's inequality, we obtain
\begin{equation}\label{3.31}
\begin{split}
\frac{\mathfrak{C}_4 }{t_0}\int_\Omega f^{\alpha_0 +\frac{p}{2} +t-1}|\nabla \eta|^2
&\leq\frac{\mathfrak{C}_6 t_0}{R^2}\int_\Omega f^{\alpha_0 +\frac{p}{2}+t-1}\eta^{\frac{2\alpha_0 + p + 2t_0 -2}{\alpha_0 + \frac{p}{2}+t_0}}\\[2mm]
&\leq\frac{\mathfrak{C}_6 t_0}{R^2}\left(\int_\Omega  f^{\alpha_0 +\frac{p}{2} +t_0}\eta^2\right)^\frac{\alpha_0 +\frac{p}{2} +t_0-1}{\alpha_0 + \frac{p}{2}+t_0}V^{\frac{1}{\alpha_0 + \frac{p}{2}+t_0}}\\[2mm]
&\leq\frac{\beta_{n,p,\alpha_0}}{2}\left[\int_\Omega  f^{\alpha_0 +\frac{p}{2} +t_0}\eta^2 +\left(\frac{2\mathfrak{C}_6 t_0 }{\beta_{n,p,\alpha_0}R^2}\right)^{\alpha_0 + \frac{p}{2}+t_0} V\right],
\end{split}
\end{equation}
where we have used
\begin{align*}
\mathfrak{C}_4 R^2 |\nabla \eta |^2\leq \mathfrak{C}_4 |C(n)|^2 \left( \alpha_0 + \frac{p}{2}+t_0 \right)^2 \eta^{\frac{2\alpha_0  + p + 2t_0 -2}{\alpha_0 + \frac{p}{2}+t_0}}\leq \mathfrak{C}_6 t_0^2 \eta^{\frac{2\alpha_0  + p + 2t_0 -2}{\alpha_0 + \frac{p}{2}+t_0}}.
\end{align*}
Substituting  \eqref{3.31} into \eqref{3.30} and \eqref{3.87},  it arrives at for $n\geq3$
\begin{align}
\left\|f^{\frac{\alpha_0 + t-1}{2}+ \frac{p}{4}} \eta\right\|^2_{L^\frac{2n}{n-2}(\Omega)}
\leq&\frac{t_0}{\mathfrak{C}_3}e^{t_0} V^{1-\frac{2}{n}} R^2\left[\frac{2\mathfrak{C}_5 t_0^2 }{R^2}\left(\frac{2\mathfrak{C}_5 t_0^2}{\beta_{n,p,\alpha_0}R^2}\right)^{\alpha_0 + t_0 + \frac{p}{2}-1}\right.\nonumber\\
&\left.+\frac{2\mathfrak{C}_6 t_0^2 }{R^2}\left(\frac{2\mathfrak{C}_6 t_0}{\beta_{n,p,\alpha_0} R^2}\right)^{\alpha_0 + t_0 + \frac{p}{2}-1}\right]
\nonumber\\[2mm]
&\leq\mathfrak{C}_7^{^{\alpha_0 + t_0 + \frac{p}{2}-1}} e^{t_0} V^{1-\frac{2}{n}} t_0^3\left(\frac{t_0^2}{R^2}\right)^{\alpha_0 + t_0 + \frac{p}{2}-1}
\label{3.45}
\end{align}
and for $n=2$
\begin{align}
\left\|f^{\frac{\alpha_0 + t-1}{2}+ \frac{p}{4}} \eta \right\|^2_{L^4(\Omega)}
\leq&\frac{t_0}{\mathfrak{C}_3}e^{t_0} V^{\frac{1}{2}} R^2 \left[\frac{2\mathfrak{C}_5 t_0^2 }{R^2}\left(\frac{2\mathfrak{C}_5 t_0^2}{\beta_{2,p,\alpha_0}R^2}\right)^{\alpha_0 + t_0 + \frac{p}{2}-1}\right.\nonumber\\
&\left. + \frac{2\mathfrak{C}_6 t_0^2 }{R^2}\left(\frac{2\mathfrak{C}_6 t_0}{ \beta_{2,p,\alpha_0} R^2 }\right)^{\alpha_0 + t_0 + \frac{p}{2}-1}\right]\nonumber\\[2mm]
\leq&\mathfrak{C}_7^{^{\alpha_0 + t_0 + \frac{p}{2}-1}} e^{t_0} V^{\frac{1}{2}} t_0^3\left(\frac{t_0^2}{R^2} \right)^{\alpha_0 + t_0 + \frac{p}{2}-1}
\label{3.46}
\end{align}
where the constant
$$
\mathfrak{C}_7 = \mathfrak{C}_7 (n,p)=\left(2\dfrac{\mathfrak{C}_5+ \mathfrak{C}_6}{\mathfrak{C}_3} +1 \right) \left(  \dfrac{2\mathfrak{C}_5 + 2\mathfrak{C}_6}{\beta_{n,p,\alpha_0}}+1\right)
$$
satisfies
\begin{align*}
\mathfrak{C}_7^{\alpha_0 + t_0 + \frac{p}{2}-1}\geq
    \frac{2\mathfrak{C}_5 }{\mathfrak{C}_3} \left(
                              \frac{2\mathfrak{C}_5 }{\beta_{n,p,\alpha_0}}
                        \right)^{\alpha_0 + t_0 + \frac{p}{2}-1}
    +
    \frac{\mathfrak{C}_6}{\mathfrak{C}_3 t_0}\left(
               \frac{2\mathfrak{C}_6 }{\beta_{n,p,\alpha_0}t_0}
               \right)^{\alpha_0 + t_0 + \frac{p}{2}-1}.
\end{align*}
Denoting
$$
\mathfrak{C}_8 := \mathfrak{C}_7 \sup\limits_{t_0 \geq 1} e^{\frac{t_0}{\alpha_0 + t_0 + \frac{p}{2}-1}}
          t_0^{\frac{3}{\alpha_0 + t_0 + \frac{p}{2}-1}}
$$
and
\begin{equation*}
\gamma :=\left\{
    \begin{aligned}
        \frac{n}{n-2} \left( \alpha_0 + \frac{p}{2}+ t_0  -1 \right), &\quad {\rm for } \:  n \geq 3,
        \\[2mm]
        2\left( \alpha_0 + \frac{p}{2}+ t_0  -1 \right) , &\quad {\rm for } \: n =2,
    \end{aligned}
    \right.
\end{equation*}
 we obtain the following result from \eqref{3.45} and \eqref{3.46}
\begin{align*}
    \left\|
                f \eta^{\frac{2}{\alpha_0 + t_0 + \frac{p}{2}-1}}
    \right\|_{L^\gamma (\Omega)}
    \leq
    \mathfrak{C}_7 e^{\frac{t_0}{\alpha_0 + t_0 + \frac{p}{2}-1}}
             t_0^{\frac{3}{\alpha_0 + t_0 + \frac{p}{2}-1}}
            \frac{ V^{\frac{1}{\gamma}}
 t_0^2}{R^2}
        \leq
        \frac{   \mathfrak{C}_8 V^{\frac{1}{\gamma}}   t_0^2}{R^2}.
\end{align*}
Here we have used
$$
t_0^{\frac{3}{\alpha_0+ t_0 + \frac{p}{2}-1}} \leq t_0^\frac{3}{t_0} \leq \max\limits_{x\geq 1} x^\frac{3}{x}= e^\frac{3}{e}.
$$
Based on $\eta \equiv 1\:  \mathrm{in } \: B_{\frac{3R}{4}}(o)$, it deduces
\begin{align*}
\| f\|_{L^\gamma \left( B_{\frac{3R}{4}}(o)\right)}\leq \frac{\mathfrak{C}_8 V^{\frac{1}{\gamma}}t_0^2}{R^2}.
\end{align*}
The proof of this lemma is completed.
\end{proof}

\subsection{ The gradient estimate  by Moser iteration: proof of Theorem \ref{theorem1}}\label{sec2.4}\

\begin{lemma}
Let $(M^n, g) $ be a complete Riemannian manifold with $\mathrm{Ric} \geq -(n-1) \kappa$ and $n\geq 2$.
If $v $ is a weak solutions to the $p $-Laplacian equation with  cubic polynomial nonlinearity \eqref{GAC} on the geodesic ball $B_R(o) \subset M $, then there exists a constant $\mathfrak{C}_{11} = \mathfrak{C}_{11}(n,p) >0$,  such that
\begin{align*}
\|f\|_{L^\infty( R_{\frac{R}{2}}(o))}\leq\mathfrak{C}_{11}  \left(  \frac{1+ \sqrt{\kappa}R}{R }\right)^2,
\end{align*}
where $V$ is the volume of the geodesic ball $B_R(o) $.
\end{lemma}

\begin{proof}
Recall \eqref{3.27} without the first term in the left hand side
\begin{align}
&\frac{\mathfrak{C}_5 t_0^2 }{R^2}  \int_{\Omega}  f^{\alpha_0 + \frac{p}{2}+t-1} \eta^2 \ +\ \frac{\mathfrak{C}_4 }{t}
\int_\Omega f^{\alpha_0 +\frac{p}{2} +t-1}  |\nabla \eta|^2 \nonumber\\[2mm]
&\geq \left\{
\begin{aligned}
\frac{\mathfrak{C}_3 }{t}    e^{-t_0} V^{\frac{2}{n}} R^{-2} \left\|f^{\frac{\alpha_0 + t-1}{2}+ \frac{p}{4}}\eta\right\|^2_{L^\frac{2n}{n-2}(\Omega)}, \quad {\rm for} \:n\geq 3, \\[2mm]
\frac{\mathfrak{C}_3 }{t}e^{-t_0}V^\frac{1}{2} R^{-2}\left\|f^{\frac{\alpha_0 + t-1}{2}+ \frac{p}{4}}\eta\right\|^2_{L^4(\Omega)},
\quad {\rm for} \: n=2.
\end{aligned}
\right.
\label{3.33}
\end{align}
Setting $r_k = \frac{R}{2} + \frac{R}{4^k}$ and $\Omega_k = B_{r_k} (o) $, and choosing the cut-off function sequence $\{ \eta_k\}_{k=1}^{\infty}\subset C^\infty(\Omega_k)$ satisfying
\begin{equation}\label{text function2}
0\leq \eta_k \leq 1, \qquad |\nabla \eta_k| \leq \frac{4^k C}{R} \ \mathrm{in }\: B_{r_{k}}(o)
\qquad \mathrm{and } \qquad
\eta_k \equiv1 \ \mathrm{in }\: B_{r_{k+1}}(o),
\end{equation}
and then replacing $\eta $ by $\eta_k$ in \eqref{3.33}, we obtain for $n\geq 3$
\begin{align*}
\mathfrak{C}_3 e^{-t_0} V^{\frac{2}{n}}
\left\|f^{\frac{\alpha_0 + t_0-1}{2}+ \frac{p}{4}} \eta_k\right\|^2_{L^\frac{2n}{n-2}(\Omega_k)}
&\leq\mathfrak{C}_5 t_0^2 t \int_{\Omega_k}  f^{\alpha_0 + \frac{p}{2}+t -1}\eta^2_k \ +\ \mathfrak{C}_4 R^2
\int_{\Omega_k} f^{\alpha_0 +\frac{p}{2} +t -1}  |\nabla \eta_k|^2\\[2mm]
&\leq\left(\mathfrak{C}_5 t_0^2 t + C^2  \mathfrak{C}_4 4^{2k}\right)\int_{\Omega_k}  f^{\alpha_0 +\frac{p}{2} +t -1}
\end{align*}
and for $n=2$
\begin{align*}
\mathfrak{C}_3 e^{-t_0}V^\frac{1}{2} R^{-2}\left\|f^{\frac{\alpha_0 + t-1}{2}+ \frac{p}{4}} \eta\right\|^2_{L^4(\Omega)}
&\leq\left(\mathfrak{C}_5 t_0^2 t + C^2 \mathfrak{C}_4 4^{2k}\right)\int_{\Omega_k}  f^{\alpha_0 +\frac{p}{2} +t -1}.
\end{align*}
\medskip

For $k \in \mathbb{N}^+,$ we set $\beta_1 : = \gamma,\ $ and
\begin{equation}
\beta_{k+1} :=
\begin{cases}
\dfrac{n }{n-2}\beta_k= \left(\dfrac{n }{n-2}\right)^{k+1} \gamma, &\quad {\rm for } \: n\geq 3, \\[3mm]
2\beta_k = 2^{k+1} \gamma, &\quad {\rm for }\: n=2,
\end{cases}
\end{equation}
 and  $t=t_k $ in such a way that
\begin{align*}
    t_k +\frac{p}{2 }+ \alpha_0 -1 =\beta_k.
\end{align*}
By
\begin{equation*}
t_k < \beta_k =\left\{
\begin{aligned}
\left(\dfrac{n }{n-2}\right)^k    \left(t_0 + \frac{p}{2} +\alpha_0 -1\right)\quad {\rm for } \: n\geq 3,\\[2mm]
2^k\left(t_0 + \frac{p}{2} +\alpha_0 -1\right)\quad {\rm for }\: n=2,
\end{aligned}
\right\}
\leq 4^{2k} \left(t_0 + \frac{p}{2} +\alpha_0 -1\right)
\end{equation*}
drew from the facts that $\alpha_0 $ is large enough and $\dfrac{n}{n-2} < 4^{2}$,
together with $t_0 \geq1$  and
$$\mathfrak{C}_9 :=\mathfrak{C}_3^{-1}  \left[   \mathfrak{C}_5 +\mathfrak{C}_5\left( \frac{p}{2} +\alpha_0 \right)+C^2\right],$$
we then have for $n\geq 3$
\begin{align*}
\mathfrak{C}_3 \left(\int_{\Omega_k} f^{\beta_{k+1}} \eta_k^\frac{2n}{n-2}\right)^\frac{n-2}{n}
&\leq e^{t_0 }V^{-\frac{2}{n}}
\left[\mathfrak{C}_5 t_0^2\left(t_0 + \frac{p}{2} +\alpha_0 -1\right)\left(\frac{n}{n-2}\right)^k + C^2 4^{2k}\right]
\int_{\Omega_k} f^{\beta_k}\\[2mm]
&\leq e^{t_0 } V^{-\frac{2}{n}} 4^{2k}\left[\mathfrak{C}_5 t_0^2\left(t_0 + \frac{p}{2} +\alpha_0 -1\right) + C^2 \right]\int_{\Omega_k}f^{\beta_k}
\\[2mm]
&\leq 2\mathfrak{C}_9 t_0^3 e^{t_0 }V^{-\frac{2}{n}} 4^{2k}\int_{\Omega_k} f^{\beta_k},
\end{align*}
and for $n=2$
\begin{align*}
\mathfrak{C}_3 \left(\int_{\Omega_k} f^{\beta_{k+1}} \eta_k^4\right)^2
&\leq e^{t_0 } V^{-\frac{1}{2}}\left[\mathfrak{C}_5 t_0^2\left(t_0 + \frac{p}{2} +\alpha_0 -1\right) 2^k + C^2 4^{2k}\right] \int_{\Omega_k} f^{\beta_k}\\[2mm]
&\leq 2\mathfrak{C}_9 t_0^3e^{t_0}V^{-\frac{1}{2}} 4^{2k}\int_{\Omega_k} f^{\beta_k}.
\end{align*}
Taking $\beta^{-1}$-th power  of the both sides of above inequality and then substituting \eqref{text function2}, it holds
\begin{align*}
\|f\|_{L^{\beta_{k+1}} (\Omega_{k+1})}\leq
\begin{cases}
\left(2\mathfrak{C}_9 t_0^3e^{t_0} V^{-\frac{2}{n}}\right)^\frac{1}{\beta_k} 4^{\frac{2k}{\beta_k}}\|f\|_{L^{\beta_{k}} (\Omega_{k})}
&\quad {\rm for } \: n\geq 3;\\[2mm]\left(2\mathfrak{C}_9 t_0^3e^{t_0} V^{-\frac{1}{2}}\right)^\frac{1}{\beta_k}4^{\frac{2k}{\beta_k}}
\|f\|_{L^{\beta_{k}} (\Omega_{k})}
&\quad {\rm for }\: n=2.
\end{cases}
\end{align*}
Since
\begin{align*}
\sum_{k=1}^\infty \frac{1}{\beta_k} = \frac{\beta_1^{-1}}{1- \frac{n-2}{n}} =\frac{n}{2\beta_1}
\qquad
\mbox{and}
\qquad
\sum_{k=1}^\infty\frac{k}{\beta_k}=\frac{n^2}{4\beta_1},
\end{align*}
letting $k \rightarrow \infty$ and choosing
$$
\mathfrak{C}_{10}= \mathfrak{C}_{10}(n , p,\mathfrak{C}_9)\geq \left( 2\mathfrak{C}_9 t_0^3 e^{t_0} \right)^\frac{n}{2\beta_1} 4^{ \sum_{k=1}^\infty\frac{2    k}{\beta_k} },
$$
it derives to for all $n \geq 2$
\begin{align*}
\|f\|_{L^{\infty} \left(B_{\frac{R}{2}}(o)\right)}\leq \left( 2\mathfrak{C}_9 t_0^3 e^{t_0} \right)^\frac{n}{2\beta_1} 4^{ \sum_{k=1}^\infty\frac{2    k}{\beta_k} } V^{-\frac{1}{\beta }}\| f\|_{L^{\beta}\left(B_{\frac{3R}{4}}(o)\right)}\leq\mathfrak{C}_{10} V^{-\frac{1}{\beta }}\| f\|_{L^{\beta} \left(B_{\frac{3R}{4}}(o)\right)}.
\end{align*}
In fact, it is easy to observe that
\begin{align*}
\mathfrak{C}_9^{ \frac{n-2}{2(\alpha_0 + t_0 +\frac{p}{2} -1) }},\quad t_0^{{ \frac{3(n-2)}{2(\alpha_0 + t_0 +\frac{p}{2} -1)}}},\quad
e^{{ \frac{(n-2)t_0}{2(\alpha_0 + t_0 +\frac{p}{2} -1)}}},\quad 16^{{ \frac{n(n-2)}{2(\alpha_0 + t_0 +\frac{p}{2} -1)}}}
\end{align*}
are all uniformly bounded, which enable an upper bound depending on $n$ and $p$.

Utilizing the estimate \eqref{3.28} in \autoref{lemma3} and $\mathfrak{C}_{11}= \mathfrak{C}_{10} \mathfrak{C}_{8} c_1$,
we finally gain the $L^\infty$ upper-estimate to the gradient function $f= |\nabla w|^2$
\begin{align*}
\| f\|_{L^{\infty} (B_{\frac{R}{2}}(o))}\leq \mathfrak{C}_{11}  \left(  \frac{1+ \sqrt{\kappa}R}{R }\right)^2.
\end{align*}
We finish the proof of the lemma as well as Theorem \ref{theorem1}.
\end{proof}

\subsection{ Proofs of Liouville theorem and Harnack inequality}\label{sec2.5}\

We will use the gradient estimate in  \autoref{theorem1} to prove \autoref{theorem3} and \autoref{theorem4}.
\medskip

\noindent\textbf{Proof of \autoref{theorem3}: }
The choice $\kappa=0$ in \autoref{theorem1} concludes that in the case $a_1 <v <a_2$, there holds
$$\dfrac{|\nabla v(x_0)|}{v(x_0)-a_1}\leq \sup\limits_{B_{\frac{R}{2}}(x_0)}\dfrac{|\nabla v|}{v-a_1}\leq C(n,p) R^{-1}, \: \,\forall\, x_0 \in M;$$
in the case $a_2 <v <a_3$,
$$\dfrac{|\nabla v(x_0)|}{a_3-v(x_0)}\leq\sup\limits_{B_{\frac{R}{2}}(x_0)}\dfrac{|\nabla v|}{a_3-v}\leq C(n,p) R^{-1}, \ \,\forall\, x_0 \in M;$$
and in the case $a_1 <v <a_3$, we have
\begin{align*}
\dfrac{( a_3-a_1)^2|\nabla v(x_0)|^2}{4[(v(x_0) - a_1)(a_3 - v(x_0))]^2} \leq \sup\limits_{B_{\frac{R}{2}}(x_0)} \dfrac{(a_3-a_1)^2|\nabla v|^2}{4[(v - a_1)(a_3 - v)]^2} \leq C(n,p) R^{-1}, \ \,\forall\, x_0 \in M.
\end{align*}
In fact, recall
$$\nabla v = L(1 - \tanh^2 w)\nabla w, \quad  v - a_1 = L(1 + \tanh w),\quad \mbox{and}\quad v - a_3 = L(\tanh w - 1)$$
in Section \ref{section2.2}, we claim
\begin{align*}
    \nabla w = \dfrac{(a_3- a_1) \nabla v }{2(v-a_1) (a_3-v)}.
\end{align*}
And then letting $R \rightarrow \infty$,  it gives
\begin{align*}
    \nabla v(x_0)=0,\quad  \,\forall\, x_0 \in M,
\end{align*}
which concludes that $v\equiv \mathrm{const}$ and $\Delta_p v=0$.
Recalling the  cubic polynomial nonlinearity in \eqref{GAC}, we observe the only three constant solutions are $ v=a_1$, $v= a_2$ and $ v=a_3$. With this insight, we claim
\begin{itemize}
    \item If $p $ satisfies \eqref{CY1} and $ a_1 <v <a_2 $ or  $a_2 <v <a_3$, the equation \eqref{GAC} admits no solution.
    \item If $p $ satisfies \eqref{CY2} and $ a_1 <v <a_3$, the equation \eqref{GAC} admits the  unique solution $v =a_2.$
\end{itemize}
Here end the proof.
\qed

\medskip
\noindent\textbf{Proof of \autoref{theorem4}: }
According to \autoref{theorem1},
\begin{align*}
    |\nabla w(x)|
    \leq C(n,p) \frac{1+ \sqrt{\kappa }R}{R},\
    \:\forall\, x \in M.
\end{align*}
As a result, by letting $R \rightarrow \infty$,
\begin{align}
      |\nabla w(x)|
      \leq  C(n,p)  \sqrt{\kappa }, \quad \forall\, x \in M.
\end{align}
For any fixed $x_0 \in M$, $\forall\, x \in M$, we choose a minimizing geodesic $\varphi(t)$ connecting $x$ and $x_0$:
\begin{align*}
    \varphi :[0,\ell] \rightarrow M, \quad \varphi (0) = x_0, \quad \varphi (\ell) = x,
\end{align*}
where  $\ell = \mathrm{dist} (x,x_0)$ is the distance between $x $ and $x_0$. It is easy to see
\begin{align}
w(x) - w(x_0) = \int_0^\ell \frac{\mathrm{d}}{\mathrm{d}t} w  \circ \varphi (t)  \mathrm{d} t\leq \int_0^\ell
\left|\frac{\mathrm{d}}{\mathrm{d}t} w \circ \varphi(t)\right| \mathrm{d} t\leq|\nabla w|| \varphi'(t)| \ell
= C(n,p) \sqrt{\kappa}\ell,
\end{align}
thus,
\begin{align}
    w(x_0) -C(n,p) \sqrt{\kappa}\ell
    \leq w(x)
    \leq     w(x_0) +C(n,p) \sqrt{\kappa}\ell.
\end{align}
\begin{itemize}
    \item For the transformation $w = -(p-1) \ln (v-a_1)$ with $a_1 <v <a_2$ in Section \ref{section2.1}, we note that  $ \ell \leq \frac{R}{2}$ and  obtain
    \begin{align*}
        e^{-C(n,p) (1 + \sqrt{\kappa} R) \slash2 }
        \leq \dfrac{v(x)-a_1}{v(x_0)-a_1}
        \leq e^{C(n,p) (1 + \sqrt{\kappa} R) \slash2 }
    \end{align*}
    and thus
    \begin{align*}
        \dfrac{v(x)-a_1}{v(y)-a_1}
        \leq e^{C(n,p) (1 + \sqrt{\kappa} R)},\quad  \forall \: x,\, y \in B_{R\slash 2}(x_0).
    \end{align*}
Similarly, for the transformation $w = -(p-1) \ln (a_3-v)$ with $a_2 <v <a_3$, there holds
\begin{align*}
\dfrac{a_3-v(x)}{a_3-v(y)}
\leq e^{C(n,p) (1 + \sqrt{\kappa} R)},\quad  \forall \: x,\, y \in B_{R\slash 2}(x_0).
\end{align*}
\item For the transformation
$$ w = \operatorname{arctanh}\left( \dfrac{v - m}{L} \right)= \dfrac{1}{2} \ln \dfrac{1 + \frac{v-m}{L}}{1 - \frac{v-m}{L}} = \dfrac{1}{2} \ln \dfrac{v-a_1}{a_3-v}$$
with $a_1 <v <a_3$ in Section \ref{section2.2},
we note that  $ \ell \leq \frac{R}{2}$ and  obtain
\begin{align*}
-2C(n,p) (1 + \sqrt{\kappa} R)\leq \ln \dfrac{v(x)-a_1}{a_3-v(x)}- \ln \dfrac{v(x_0)-a_1}{a_3-v(x_0)}\leq 2C(n,p) (1 + \sqrt{\kappa} R),
\end{align*}
which deduces to
\begin{align*}
\left| \ln \frac{\big(v(x)-a_1\big)\big(a_3-v(y)\big)}{\big(v(y)-a_1\big)\big(a_3-v(x)\big)} \right| \leq 4C(n,p)(1+\sqrt{\kappa}R),
\end{align*}
i.e.,
\begin{align*}
\frac{\big(v(x)-a_1\big)\big(a_3-v(y)\big)}{\big(v(y)-a_1\big)\big(a_3-v(x)\big)} \leq e^{4C(n,p) (1 + \sqrt{\kappa} R)},\quad  \forall \: x, y \in B_{R\slash 2}(x_0).
\end{align*}
\end{itemize}
This completes the proof.
\qed



\bigskip
\bigskip
\noindent {\textbf{Data Availability}} \quad The manuscript has no associated data.

\medskip
\noindent{\textbf{Declarations}

\medskip
\noindent Conflicts of Interest\quad The authors declare that they have no Conflict of interest.

\bibliographystyle{acm}

\end{document}